\documentclass[oneside,english]{amsart}
\usepackage{mathpazo}

\usepackage[T1]{fontenc}
\usepackage[latin9]{inputenc}
\usepackage{geometry}
\usepackage{color}
\usepackage{babel}
\usepackage{amstext}
\usepackage{amsthm}
\usepackage{amssymb}
\usepackage{graphicx}
\usepackage[unicode=true,pdfusetitle,
 bookmarks=true,bookmarksnumbered=false,bookmarksopen=false,
 breaklinks=false,pdfborder={0 0 1},backref=false,colorlinks=true]
 {hyperref}

\makeatletter

\providecommand{\tabularnewline}{\\}

\numberwithin{equation}{section}
\numberwithin{figure}{section}
\numberwithin{table}{section}
\usepackage{enumitem}		
\theoremstyle{plain}
\newtheorem{thm}{\protect\theoremname}[section]
\theoremstyle{definition}
\newtheorem{defn}[thm]{\protect\definitionname}
\theoremstyle{remark}
\newtheorem{rem}[thm]{\protect\remarkname}
\theoremstyle{plain}
\newtheorem{lem}[thm]{\protect\lemmaname}
\ifx\proof\undefined
\newenvironment{proof}[1][\protect\proofname]{\par
\normalfont\topsep6\p@\@plus6\p@\relax
\trivlist
\itemindent\parindent
\item[\hskip\labelsep\scshape #1]\ignorespaces
}{%
\endtrivlist\@endpefalse
}
\providecommand{\proofname}{Proof}
\fi
\theoremstyle{plain}
\newtheorem{prop}[thm]{\protect\propositionname}
\theoremstyle{plain}
\newtheorem{cor}[thm]{\protect\corollaryname}
\theoremstyle{plain}
\newtheorem{conjecture}[thm]{\protect\conjecturename}

\date{}

\usepackage{hyperref}
\hypersetup{colorlinks=true,citecolor=blue}

\usepackage{color}
\definecolor{green}{RGB}{0, 180, 0}

\newcommand{\tam}{\mathcal{T}}
\newcommand{\symgrp}{\mathcal{S}}
\newcommand{\full}{\mathcal{S}}
\newcommand{\chain}{\mathcal{C}}
\newcommand{\nofull}{\mathcal{N}}

\makeatother

\providecommand{\conjecturename}{Conjecture}
\providecommand{\corollaryname}{Corollary}
\providecommand{\definitionname}{Definition}
\providecommand{\lemmaname}{Lemma}
\providecommand{\propositionname}{Proposition}
\providecommand{\remarkname}{Remark}
\providecommand{\theoremname}{Theorem}

\begin{document}

\title{A recursion on maximal chains in the Tamari lattices}

\author{Luke Nelson}
\begin{abstract}
The Tamari lattices have been intensely studied since their introduction
by Dov Tamari around 1960. However oddly enough, a formula for the
number of maximal chains is still unknown. This is due largely to
the fact that maximal chains in the $n$-th Tamari lattice $\tam_{n}$
range in length from $n-1$ to $\binom{n}{2}$. In this note, we treat
vertices in the lattice as Young diagrams and identify maximal chains
as certain tableaux. For each $i\geq-1$, we define $\chain_{i}(n)$
as the set of maximal chains in $\tam_{n}$ of length $n+i$. We give
a recursion for $\#\chain_{i}(n)$ and an explicit formula based on
predetermined initial values. The formula is a polynomial in $n$
of degree $3i+3$. For example, the number of maximal chains of length
$n$ in $\tam_{n}$ is $\#\chain_{0}(n)=\binom{n}{3}$. The formula
has a combinatorial interpretation in terms of a special property
of maximal chains.
\end{abstract}


\maketitle

\section{\label{sec:Intro}Introduction}

The Tamari lattices $\{\tam_{n}\}$ have been intensely studied since
their introduction by Dov Tamari \cite{Tam62}. He defined $\tam_{n}$
on bracketings of a set of $n+1$ objects, with a cover relation based
on the associativity rule in one direction. Friedman and Tamari later
proved the lattice property \cite{FriedTam67}. $\tam_{n}$ is both
a quotient and a sublattice of the weak order on the symmetric group
$\symgrp_{n}$, and its Hasse diagram is the $1$-skeleton of the
associahedron (or Stasheff polytope) \cite{BjornerWachs97,Reading06}. 

As usual, $[n]$ denotes the set $\{1,\ldots,n\}$ of the first $n$
positive integers, and we adopt the convention that $[0]$ denotes
the empty set. The number of vertices in $\tam_{n}$ is the $n$-th
Catalan number, $C_{n}=\frac{1}{n+1}\binom{2n}{n}$. Triangulations
of a convex $(n+2)$-gon, noncrossing partitions of $[n]$, and Dyck
paths of length $2n$ are examples of over $200$ combinatorial structures
counted by the Catalan sequence \cite{Stanley2015}. The Kreweras
and Stanley lattices are two other noted lattices defined on Catalan
sets. The Stanley lattice is a refinement of the Tamari lattice, which
is a refinement of the Kreweras lattice \cite[Exercises 7.2.1.6-27, 28]{Knu06}
\cite{BB09}.

The search for enumeration formulas, whether for maximal chains or
intervals, \textit{etc.}, are classic problems for any family of posets.
The pursuit of solutions often leads to relationships with other combinatorial
structures and a better understanding of the poset at hand. 

Definitions concerning posets may be found in \cite[Chapter 3]{Sta12}.
If $x\leq y$ in a poset $P$, the subposet $[x,y]=\{z\in P\mid x\leq z\leq y\}$
of $P$ is called a (closed) \textit{\textcolor{green}{interval}}.
Intervals in the Kreweras lattice are in bijection to ternary trees
\cite{Kreweras72,EP82}, while those in the Stanley lattice are pairs
of noncrossing Dyck paths \cite{SC85}. Chapoton enumerated the intervals
in $\tam_{n}$, finding this to be the number of planar triangulations
(\textit{i.e.}, maximal planar graphs) \cite{Chap05}. Bergeron and
Pr\'eville-Ratelle generalized the Tamari posets to the $m$-Tamari
posets $\tam_{n}^{(m)}$ \cite{BP12} (the case $m=1$ is $\{\tam_{n}\}$).
Shortly after, Bousquet-M\'elou, Fusy and Pr\'eville-Ratelle proved the
$m$-Tamari posets are lattices and generalized Chapoton's formula
to $\tam_{n}^{(m)}$ \cite{BFP11}. Its formula has a simple factorized
form (as in Chapoton's formula) and a combinatorial interpretation;
see \cite{BB09,ChapChatPons13,ChatelPons15}.

A subset of a poset in which any two elements are comparable is called
a \textit{\textcolor{green}{chain}}. The \textit{\textcolor{green}{length}}
of a finite chain $C$ is the number of its elements minus one, denoted
$l(C)$. If $x<y$ and there does not exist $z$ such that $x<z<y$,
then we say that $y$ \textit{\textcolor{green}{covers}} $x$, or
$x$ is covered by $y$, and denote $y\gtrdot x$, or $x\lessdot y$.
A \textit{\textcolor{green}{saturated chain}} in a finite poset is
a sequence of elements $x_{0}\lessdot x_{1}\lessdot\cdots\lessdot x_{l-1}\lessdot x_{l}$.
If a poset has an element $x$ with the property that $x\leq y$ for
all $y$ in the poset, we denote that element by \textit{\textcolor{green}{$\hat{0}$}}.
Similarly, \textit{\textcolor{green}{$\hat{1}$}} is the element above
all others if it exists. In a finite poset with a $\hat{0}$ and $\hat{1}$,
a \textit{\textcolor{green}{maximal chain}} is a sequence of elements
$\hat{0}=x_{0}\lessdot x_{1}\lessdot\cdots\lessdot x_{l-1}\lessdot x_{l}=\hat{1}$. 

Maximal chains in the Kreweras lattice, of noncrossing partitions
of $[n]$, are in bijection to factorizations of an $n$-cycle as
the product of $n-1$ transpositions \cite{Kreweras72} \cite[Exercise 7.2.1.6-33]{Knu06},
the number of which is $n^{n-2}$ \cite{Denes59}. This is also the
number of parking functions of length $n-1$ \cite{FoataRiordan1974}
\cite{Stanley97}, and the number of trees on $n$ labeled vertices
\cite{CayleyTreeTheorem}. There is a simple bijection between maximal
chains in the Stanley lattice of order $n$ and standard Young tableaux
of shape $(n-1,n-2,\ldots,1)$ (see \cite[Exercise 7.2.1.6-34]{Knu06}
for one), whose number is given by the hook length formula \cite{FraRobThr54}:
\begin{equation}
\frac{\binom{n}{2}!}{1^{n-1}3^{n-2}\cdots(2n-5)^{2}(2n-3)^{1}}.\label{eq:FormulaStanleyMaximalChains}
\end{equation}
This is again the number of maximal chains in the weak order on $\symgrp_{n}$
\cite{EdelmanGreene87,Sta84}.

A natural question arises: What is the number of maximal chains in
the Tamari lattice? Quoting Knuth, ``The enumeration of such paths
in Tamari lattices remains mysterious.'' Despite their interesting
aspects and the attention they have received, a formula for the number
of maximal chains in the Tamari lattices is still unknown. The complexity
of the problem is largely due to the fact that maximal chains in the
Tamari lattices vary over a great range. The longest chains in $\tam_{n}$
have length $\binom{n}{2}$ and the unique shortest has length $n-1$
\cite{Markowsky92} \cite[Exercise 7.2.1.6-27(h)]{Knu06}. We determined
a bijection between longest chains in $\tam_{n}^{(m)}$ and standard
$m$-shifted tableaux of shape $(m(n-1),m(n-2),\ldots,m)$ in \cite{FN13}
and, as a corollary, a formula for the number of longest chains in
$\tam_{n}$:
\begin{equation}
\binom{n}{2}!\frac{(n-2)!(n-3)!\cdots(2)!(1)!}{(2n-3)!(2n-5)!\cdots(3)!(1)!}.\label{eq:FormulaTamLongestChains}
\end{equation}

Keller introduced \textit{green mutations} and maximal sequences of
such mutations, called \textit{maximal green sequences} \cite{Keller2011}.
In certain cases, maximal green sequences are in bijection with maximal
chains in the Tamari lattice or the Cambrian lattice (a generalization
of the Tamari lattice \cite{Reading06}) \cite{Keller2013,GreenSequences}.
Garver and Musiker list several applications of maximal green sequences
to representation theory and physics \cite{GreenSequences}. The problems
of enumeration and classification of such sequences are noted interests.

Aside from these cases, we are unaware of any results pertaining to
the enumeration of maximal chains in the Tamari lattices. In this
note, our focus pertains to the following definition.
\begin{defn}
\label{def:Cin}Let $i\geq-1$ and $n\geq1$. \textit{\textcolor{green}{$\chain_{i}(n)$}}
is the set of all maximal chains of length $n+i$ in $\tam_{n}$.
\end{defn}
\begin{table}[h]
\begin{centering}
\begin{tabular}{|c|c|c|c|c|c|c|c|c|c|}
\hline 
{\footnotesize{}Length} & {\footnotesize{}$\tam_{1}$} & {\footnotesize{}$\tam_{2}$} & {\footnotesize{}$\tam_{3}$} & {\footnotesize{}$\tam_{4}$} & {\footnotesize{}$\tam_{5}$} & {\footnotesize{}$\tam_{6}$} & {\footnotesize{}$\tam_{7}$} & {\footnotesize{}$\tam_{8}$} & {\footnotesize{}$\tam_{9}$}\tabularnewline
\hline 
\hline 
{\footnotesize{}n - 1} & {\footnotesize{}1} & {\footnotesize{}1} & {\footnotesize{}1} & {\footnotesize{}1} & {\footnotesize{}1} & {\footnotesize{}1} & {\footnotesize{}1} & {\footnotesize{}1} & {\footnotesize{}1}\tabularnewline
\hline 
{\footnotesize{}n} &  &  & {\footnotesize{}1} & {\footnotesize{}4} & {\footnotesize{}10} & {\footnotesize{}20} & {\footnotesize{}35} & {\footnotesize{}56} & {\footnotesize{}84}\tabularnewline
\hline 
{\footnotesize{}n + 1} &  &  &  & {\footnotesize{}2} & {\footnotesize{}22} & {\footnotesize{}112} & {\footnotesize{}392} & {\footnotesize{}1,092} & {\footnotesize{}2,604}\tabularnewline
\hline 
{\footnotesize{}n + 2} &  &  &  & {\footnotesize{}2} & {\footnotesize{}22} & {\footnotesize{}232} & {\footnotesize{}1,744} & {\footnotesize{}9,220} & {\footnotesize{}37,444}\tabularnewline
\hline 
{\footnotesize{}n + 3} &  &  &  &  & {\footnotesize{}18} & {\footnotesize{}382} & {\footnotesize{}4,474} & {\footnotesize{}40,414} & {\footnotesize{}280,214}\tabularnewline
\hline 
{\footnotesize{}n + 4} &  &  &  &  & {\footnotesize{}13} & {\footnotesize{}348} & {\footnotesize{}8,435} & {\footnotesize{}123,704} & {\footnotesize{}1,321,879}\tabularnewline
\hline 
{\footnotesize{}n + 5} &  &  &  &  & {\footnotesize{}12} & {\footnotesize{}456} & {\footnotesize{}12,732} & {\footnotesize{}276,324} & {\footnotesize{}4,578,596}\tabularnewline
\hline 
{\footnotesize{}n + 6} &  &  &  &  &  & {\footnotesize{}390} & {\footnotesize{}17,337} & {\footnotesize{}550,932} & {\footnotesize{}12,512,827}\tabularnewline
\hline 
{\footnotesize{}n + 7} &  &  &  &  &  & {\footnotesize{}420} & {\footnotesize{}21,158} & {\footnotesize{}917,884} & {\footnotesize{}29,499,764}\tabularnewline
\hline 
{\footnotesize{}n + 8} &  &  &  &  &  & {\footnotesize{}334} & {\footnotesize{}27,853} & {\footnotesize{}1,510,834} & {\footnotesize{}62,132,126}\tabularnewline
\hline 
{\footnotesize{}n + 9} &  &  &  &  &  & {\footnotesize{}286} & {\footnotesize{}33,940} & {\footnotesize{}2,166,460} & {\footnotesize{}120,837,274}\tabularnewline
\hline 
{\footnotesize{}n + 10} &  &  &  &  &  &  & {\footnotesize{}41,230} & {\footnotesize{}3,370,312} & {\footnotesize{}221,484,557}\tabularnewline
\hline 
{\footnotesize{}n + 11} &  &  &  &  &  &  & {\footnotesize{}45,048} & {\footnotesize{}4,810,150} & {\footnotesize{}393,364,848}\tabularnewline
\hline 
{\footnotesize{}n + 12} &  &  &  &  &  &  & {\footnotesize{}50,752} & {\footnotesize{}7,264,302} & {\footnotesize{}666,955,139}\tabularnewline
\hline 
{\footnotesize{}n + 13} &  &  &  &  &  &  & {\footnotesize{}41,826} & {\footnotesize{}10,435,954} & {\footnotesize{}1,134,705,692}\tabularnewline
\hline 
{\footnotesize{}n + 14} &  &  &  &  &  &  & {\footnotesize{}33,592} & {\footnotesize{}15,227,802} & {\footnotesize{}1,933,708,535}\tabularnewline
\hline 
{\footnotesize{}n + 15} &  &  &  &  &  &  &  & {\footnotesize{}20,089,002} & {\footnotesize{}3,316,121,272}\tabularnewline
\hline 
{\footnotesize{}n + 16} &  &  &  &  &  &  &  & {\footnotesize{}27,502,220} & {\footnotesize{}5,604,687,775}\tabularnewline
\hline 
{\footnotesize{}n + 17} &  &  &  &  &  &  &  & {\footnotesize{}32,145,952} & {\footnotesize{}9,577,349,974}\tabularnewline
\hline 
{\footnotesize{}n + 18} &  &  &  &  &  &  &  & {\footnotesize{}36,474,460} & {\footnotesize{}15,969,449,634}\tabularnewline
\hline 
{\footnotesize{}n + 19} &  &  &  &  &  &  &  & {\footnotesize{}30,474,332} & {\footnotesize{}26,387,217,370}\tabularnewline
\hline 
{\footnotesize{}n + 20} &  &  &  &  &  &  &  & {\footnotesize{}23,178,480} & {\footnotesize{}41,902,119,016}\tabularnewline
\hline 
{\footnotesize{}n + 21} &  &  &  &  &  &  &  &  & {\footnotesize{}65,076,754,954}\tabularnewline
\hline 
{\footnotesize{}n + 22} &  &  &  &  &  &  &  &  & {\footnotesize{}93,803,013,648}\tabularnewline
\hline 
{\footnotesize{}n + 23} &  &  &  &  &  &  &  &  & {\footnotesize{}131,664,410,706}\tabularnewline
\hline 
{\footnotesize{}n + 24} &  &  &  &  &  &  &  &  & {\footnotesize{}158,363,393,996}\tabularnewline
\hline 
{\footnotesize{}n + 25} &  &  &  &  &  &  &  &  & {\footnotesize{}179,041,479,392}\tabularnewline
\hline 
{\footnotesize{}n + 26} &  &  &  &  &  &  &  &  & {\footnotesize{}150,158,648,356}\tabularnewline
\hline 
{\footnotesize{}n + 27} &  &  &  &  &  &  &  &  & {\footnotesize{}108,995,910,720}\tabularnewline
\hline 
\hline 
{\footnotesize{}Totals} & {\footnotesize{}1} & {\footnotesize{}1} & {\footnotesize{}2} & {\footnotesize{}9} & {\footnotesize{}98} & {\footnotesize{}2,981} & {\footnotesize{}340,549} & {\footnotesize{}216,569,887} & {\footnotesize{}994,441,978,397}\tabularnewline
\hline 
\end{tabular}
\par\end{centering}

\caption{$\#\chain_{i}(n)$: Number of maximal chains in $\tam_{n}$ of length
$n+i$}
\label{tab:1-1}

\end{table}

The main result of this note is Theorem \ref{thm:MainThrm}: We give
a recursion for $\#\chain_{i}(n)$ and an explicit formula based on
predetermined initial values. The formula is a polynomial in $n$
of degree $3i+3$. For example, the number of maximal chains of length
$n-1$ in $\tam_{n}$ is $\#\chain_{-1}(n)=1$, while the number of
length $n$ is $\#\chain_{0}(n)=\binom{n}{3}$. Table \ref{tab:1-1}
is a computer based compilation of the numbers of maximal chains by
length in $\tam_{n}$ up through $\tam_{9}$. The numbers of lengths
$3,4,5$ and $6$ in $\tam_{4}$ are $\#\chain_{-1}(4)=1$, $\#\chain_{0}(4)=4$,
$\#\chain_{1}(4)=2$ and $\#\chain_{2}(4)=2$, respectively. The bottom
entry of a column is the number of longest chains in $\tam_{n}$ (given
by equation (\ref{eq:FormulaTamLongestChains})). For example, the
number of longest chains in $\tam_{6}$ is $286$.

Bernardi and Bonichon rewrote the covering relation in $\tam_{n}$
in terms of Dyck paths \cite{BB09}. We find it useful to work mainly
from the perspective of Young diagrams, but rely on properties of
both sets. We present basic terminology and the covering relation
in Section \ref{sec:Prelim}. 

We rely on two main maps: $\psi$ and $\phi_{i,n}^{r}$. We use $\psi$
to identify maximal chains in $\tam_{n}$ with certain tableaux. We
obtain an expression for the number of maximal chains using $\phi_{i,n}^{r}$,
which takes a maximal chain in $\chain_{i}(n)$ to one in $\chain_{i}(n+1)$.
Because of $\psi$, we may express $\phi_{i,n}^{r}$ as a map on tableaux.
In Section \ref{sec:psi}, we define $\psi$, establish basic properties,
and conclude with Theorem \ref{thrm:Alpha}, by defining the first
part of $\phi_{i,n}^{r}$. A maximal chain in the image of $\psi$
may or may not possess a ``plus-full-set''; see Definition \ref{def:FS}.
In Section \ref{sec:phi}, we establish the other parts of $\phi_{i,n}^{r}$
in Theorems \ref{thm:Beta} and \ref{thrm:PrePhi}. We specialize
Theorem \ref{thrm:PrePhi} to Theorem \ref{thrm:PhiINR}, thereby
defining $\phi_{i,n}^{r}$, where $r$ determines the domain and codomain
in terms of plus-full-sets. The main focus of Section \ref{sec:phi}
and a key ingredient leading up to our main objective is the fact
that this map is bijective.

In Section \ref{sec:formula}, we gather more on properties and consequences
of $\phi_{i,n}^{r}$ and tie our results together to write a recursive
formula for $\#\chain_{i}(n)$. $\phi_{i,n}^{r}$ takes a maximal
chain to one with one more plus-full-set; see Proposition \ref{prop:NumPFS}.
This enables us to write every maximal chain, which has a plus-full-set,
uniquely in terms of one with one less plus-full-set. We extend this
to a unique representation in terms of a maximal chain with no plus-full-sets
in Corollary \ref{cor:UniqueRepresentationByPhi}. By relating this
representation to specific plus-full-sets that a maximal chain contains
in Proposition \ref{prop:TtupleRelatesToPFS}, we obtain an expression
for $\#\chain_{i}(n)$; see equation (\ref{eq:=000023Ci(n)FirstExpr}).
For each $i\geq-1$, there exists a maximal chain in $\chain_{i}(2i+3)$
containing no plus-full-sets (Lemma \ref{lem:2iPlus3}), but surprisingly,
for all $n\geq2i+4$, each maximal chain in $\chain_{i}(n)$ has a
plus-full-set (Theorem \ref{thrm:GOrE2IPlus4}). We utilize these
latter two facts to refine our expression for $\#\chain_{i}(n)$ and
achieve our main objective in Theorem \ref{thm:MainThrm}.

\section{\label{sec:Prelim}Preliminaries}

In this section, we present basic terminology of Young diagrams and
Dyck paths and the covering relation in the Tamari lattice. 

A \textit{\textcolor{green}{partition}} of a positive integer $l$
is a weakly decreasing sequence $\lambda=(\lambda_{1},\lambda_{2},\ldots,\lambda_{k})$
of positive integers summing to $l$. A \textit{\textcolor{green}{Young
diagram}} $Y$ of \textit{\textcolor{green}{shape}} $\lambda$ is
a left-justified collection of boxes having $\lambda_{j}$ boxes in
row $j$, for $1\leq j\leq k$. The shape of $Y$ is denoted \textit{\textcolor{green}{$sh(Y)$}}.
Rows and columns of the diagram begin with an index of one. The \textit{\textcolor{green}{length}}
of a row (or of a column) is its number of boxes. We denote the box
in row $x$ and column $y$ by $(x,y)$. We adopt the \textit{English
notation}, in which rows are indexed downward. The \textit{\textcolor{green}{empty
partition}} $\lambda=(0)$ is associated with the \textit{\textcolor{green}{null
diagram}} $\emptyset$ having no boxes. The \textit{\textcolor{green}{staircase
shape}} $(n,n-1,n-2,\ldots,1)$ is denoted $\delta_{n}$, where we
set $\delta_{n}=(0)$ if $n\leq0$. Often times, which will be clear
by the context, we abuse notation by identifying a partition $\lambda$
with its associated Young diagram also denoted $\lambda$ or vice
versa. For example, $\delta_{3}$ is the shape $(3,2,1)$ or it is
the Young diagram of that shape. 

A \textit{\textcolor{green}{Dyck path}} of length $2n$ is a path
on the square grid of north and east steps from $(0,0)$ to $(n,n)$
which never goes below the line $y=x$. Necessarily, every Dyck path
begins with a north step and ends with an east step, and has an equal
number of both types of steps. The \textit{\textcolor{green}{height}}
of a Dyck path is its number of north steps. In \cite{BB09}, vertices
in $\tam_{n}$ are interpreted as the set of Dyck paths of length
$2n$, the number of which is the $n$-th Catalan number $C_{n}=\frac{1}{n+1}\binom{2n}{n}$.

There is a natural bijective correspondence between the set of Dyck
paths of length $2n$ and a set of Young diagrams, to which we identify
the set of vertices in $\tam_{n}$: Roughly speaking, a Dyck path
gives the silhouette of the Young diagram. This is the set of Young
diagrams contained in $\delta_{n-1}$. $\tam_{1}$ is comprised of
a single vertex, the null diagram. Figure \ref{fig:2-1} is the set
of $C_{4}=14$ Dyck paths of length $8$ and corresponding Young diagrams.

\begin{figure}[h]
\begin{centering}
\includegraphics[scale=0.5]{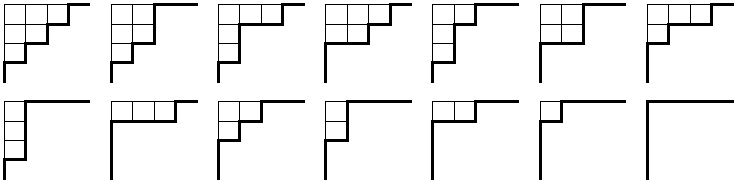}
\par\end{centering}

\caption{\label{fig:2-1}The vertices of $\tam_{4}$ in terms of Dyck paths
and Young diagrams}
\end{figure}

\begin{rem}
\label{rem:DyckPathYoungDiagramCorrespondence}Let $Y\in\tam_{n}$
be a Young diagram. For each $m\geq n$, $Y\in\tam_{m}$ and $Y$
corresponds to exactly one Dyck path of length $2m$. In the other
direction, any Dyck path (regardless of length) corresponds to exactly
one Young diagram.\end{rem}
\begin{defn}
\label{def:PrimePathSlantLine}If $P$ is a Dyck path and $L$ is
the line segment (of slope one) that joins the endpoints of $P$,
then $P$ is said to be \textit{\textcolor{green}{prime}} if $P$
intersects $L$ only at the endpoints of $P$. (We word the definition
of prime Dyck paths differently than in \cite{BB09}, but it has the
same meaning.)
\end{defn}
Of the Dyck paths in Figure \ref{fig:2-1}, only the last five in
the second row are prime, corresponding to the Young diagrams contained
in $(2,1)$. A Dyck path of length $2n$ has exactly $n$ \textit{\textcolor{green}{prime
Dyck subpaths}}, each uniquely determined by its beginning north step.
In Figure \ref{fig:2-2}, for the given Dyck path of length $8$,
its $4$ prime Dyck subpaths are bolded. The line segment joining
the endpoints of each prime Dyck subpath is drawn. As required in
Proposition \ref{prop:DyckPathsDecreasingHeights}, we characterize
pairs of prime Dyck subpaths. 
\begin{lem}
\label{lem:PrimePathsChar}If $Q$ and $R$ are two prime Dyck subpaths
of a Dyck path, then exactly one of the following characterizes $Q$
and $R$:
\begin{enumerate}
\item \label{enu:PrimePathsChar1}$Q\cap R=\emptyset$, i.e., $Q$ and $R$
have no common points.
\item \label{enu:PrimePathsChar2}$Q$ and $R$ intersect in a single point.
\item \label{enu:PrimePathsChar3}$Q\subsetneq R$ or $R\subsetneq Q$,
i.e., $Q$ is a proper subpath of $R$, or $R$ is a proper subpath
of $Q$.
\item \label{enu:PrimePathsChar4}$Q=R$.
\end{enumerate}
\end{lem}
\begin{proof}
Suppose neither (\ref{enu:PrimePathsChar1}) nor (\ref{enu:PrimePathsChar2}).
Then $Q$ and $R$ must have a step in common. If the endpoints of
$Q$ and of $R$ all lie on the same line, then $Q=R$. If, without
loss of generality, the line containing the endpoints of $Q$ lies
above the line containing the endpoints of $R$, then $Q$ is a proper
subpath of $R$.
\end{proof}
\begin{figure}[h]
\raggedleft{}%
\begin{minipage}[c]{0.5\columnwidth}%
\begin{center}
\includegraphics[scale=0.5]{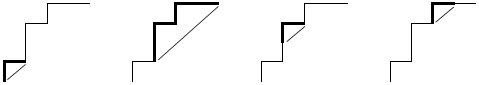}
\par\end{center}

\begin{center}
\caption{\label{fig:2-2}Prime Dyck subpaths}

\par\end{center}%
\end{minipage}\hfill{}%
\begin{minipage}[c]{0.45\columnwidth}%
\begin{center}
\includegraphics[scale=0.5]{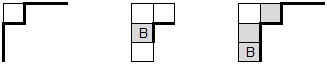}
\par\end{center}

\begin{center}
\caption{\label{fig:2-3}Prime paths}

\par\end{center}%
\end{minipage}
\end{figure}

For an example of Lemma \ref{lem:PrimePathsChar}(\ref{enu:PrimePathsChar1}),
take $Q$ and $R$ to be the prime Dyck subpaths in the first and
third examples of Figure \ref{fig:2-2}. For \ref{lem:PrimePathsChar}(\ref{enu:PrimePathsChar2}),
let $Q$ and $R$ be the subpaths in the first and second examples.
For $Q\subsetneq R$, let $Q$ and $R$ be the subpaths in the third
and second examples, respectively. 

The notion of prime is intimately tied to the covering relation in
the Tamari lattices. We need to extend this notion.
\begin{defn}
\label{def:PrimePathOfBoxOrRowCB}Let $Y$ be a Young diagram and
$d\geq1$. Let $e$ be the vertical edge at the end of row $d$ in
$Y$ for which $e$ is on a corresponding Dyck path $P$ to $Y$ (row
$d$ may be empty). The \textit{\textcolor{green}{prime path of row
$d$}} is the prime Dyck subpath of $P$ beginning with $e$. 

Suppose $B$ is the last box of its row in $Y$. The \textit{\textcolor{green}{prime
path of $B$}} is the prime path of its row. The \textit{\textcolor{green}{$B$-strip}}
is the set of all boxes in $Y$ with the right vertical edge on the
prime path of $B$. If $B$ is the last box of its row and the lowest
box of its column, then $B$ is a \textit{\textcolor{green}{corner
box}}. 
\end{defn}
In Figure \ref{fig:2-3}, prime paths are bolded, and $B$-strips
are grayed. The first example is the prime path of the third row in
the Young diagram of shape $(1)$. Each of the second and third examples
is the prime path of a box $B$ in the Young diagram of shape $(2,1,1)$.
$B$ is a corner box in the third example.

We give two versions of the covering relation. The second, in terms
of Young diagrams, follows from the correspondence of Dyck paths. 
\begin{prop}
\label{prop:CoveringDyckPaths}\cite[Proposition 2.1]{BB09} Covering
relation in the Tamari lattices: Dyck paths. Let $P$ and $P'$ be
Dyck paths. Then $P'$ covers $P$ ($P'\gtrdot P$) if and only if
there exists an east step $e$ in $P$ followed by a north step, such
that $P'$ is obtained from $P$ by swapping $e$ and the prime Dyck
subpath following it.
\end{prop}
\begin{figure}[h]
\begin{centering}
$P$\hspace{0.6in}$P'$\hspace{0.65in}$P$\hspace{0.85in}$P'$\hspace{0.5in}
\par\end{centering}

\begin{centering}
\includegraphics[scale=0.5]{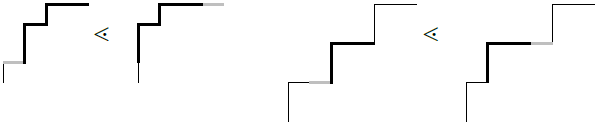}
\par\end{centering}

\caption{\label{fig:2-4}Covering relation in the Tamari lattices: Dyck Paths}
\end{figure}

\begin{prop}
\label{prop:CoveringYoungDiagrams}Covering relation in the Tamari
lattices: Young diagrams. Let $Y$ and $Y'$ be Young diagrams. Then
$Y'$ covers $Y$ ($Y'\gtrdot Y$) if and only if there exists a corner
box $B$ in $Y$, such that $Y'$ is obtained from $Y$ by removing
the $B$-strip. 
\end{prop}
\begin{figure}[h]
\begin{centering}
$Y$\hspace{0.6in}$Y'$\hspace{0.65in}$Y$\hspace{0.85in}$Y'$\hspace{0.5in}
\par\end{centering}

\begin{centering}
\includegraphics[scale=0.5]{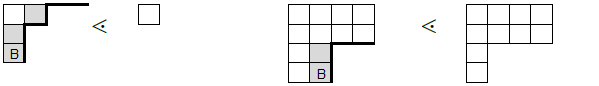}
\par\end{centering}

\caption{\label{fig:2-5}Covering relation in the Tamari lattices: Young diagrams}
\end{figure}

Examples of the covering relation for Dyck paths in Figure \ref{fig:2-4}
correspond to the examples for Young diagrams in Figure \ref{fig:2-5}.
In the Dyck path examples, the east step of $P$ referenced in the
proposition is grayed, and the prime Dyck subpath following it is
bolded. In the Young diagram examples, the prime path of the corner
box $B$ in $Y$ is bolded, and the $B$-strip is grayed. 

\begin{figure}[h]
\begin{minipage}[c]{0.35\columnwidth}%
\begin{center}
$\tam_{3}$
\par\end{center}

\begin{center}
\includegraphics[scale=0.5]{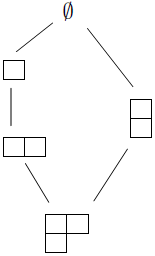}
\par\end{center}%
\end{minipage}\hfill{}%
\begin{minipage}[c]{0.55\columnwidth}%
\begin{center}
$\tam_{4}$
\par\end{center}

\begin{center}
\includegraphics[scale=0.5]{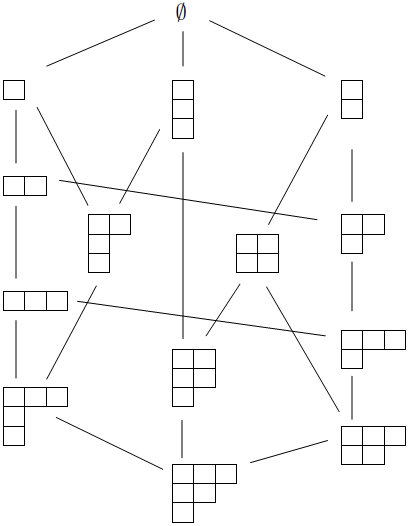}
\par\end{center}%
\end{minipage}\caption{\label{fig:2-6}Hasse diagrams for $\tam_{3}$ and $\tam_{4}$ in
terms of Young diagrams}
\end{figure}

The Hasse diagrams for $\tam_{3}$ and $\tam_{4}$ are shown in Figure
\ref{fig:2-6}. The maximum element $\hat{1}$ in $\tam_{n}$ is the
null diagram, and the minimum element $\hat{0}$ is $\delta_{n-1}$.

\section{\label{sec:psi}Representation of maximal chains}

In this section, we relate an efficient approach to work with certain
saturated chains in the Tamari lattices through the map $\psi$. In
particular, $\psi$ assigns each maximal chain to a unique tableau.
We establish basic properties and explicitly characterize chains in
terms of tableaux. We then enter into more technical material and
conclude with Theorem \ref{thrm:Alpha}, where we define the first
piece of $\phi_{i,n}^{r}$.
\begin{defn}
\label{def:Tableau}For our purposes, a \textit{\textcolor{green}{tableau}}
$T$ of shape $\lambda$ is obtained by filling each box of the Young
diagram of shape $\lambda$ with a positive integer, where each row
strictly increases when read left to right, and each column weakly
increases when read top to bottom. The \textit{\textcolor{green}{length}}
of $T$, denoted $l(T)$, is the number of its distinct labels. We
also require that its set of labels is precisely $[l(T)]$. (Often
in the literature, this is the definition of a \textit{row-strict
tableau}.)

For $r\in[l(T)]$, the \textit{\textcolor{green}{$r$-set}} is the
set of all boxes in $T$ labeled with $r$. The label in the box $(x,y)$
is denoted $T(x,y)$. For $r\in\{0,1,\ldots,l(T)\}$, the tableau
obtained from $T$ made of all the elements less than or equal to
$r$ is denoted \textit{\textcolor{green}{$T^{(r)}$}}. 
\end{defn}

\begin{defn}
\label{def:Psi}Let $C=(\emptyset=Y_{0}\gtrdot Y_{1}\gtrdot\cdots\gtrdot Y_{l})$
be a saturated chain under the Tamari order in terms of Young diagrams.
As $C$ is traversed upwards in the Hasse diagram, boxes are removed
from $Y_{l}$. For each $r\in[l]$, label the boxes removed in the
transition $Y_{r-1}\gtrdot Y_{r}$ with $r$. The resulting tableau,
of the shape of $Y_{l}$, is $\psi(C)$. A \textit{\textcolor{green}{$\psi$-tableau}}
is an element in the image of $\psi$.
\end{defn}
Examples of $\psi$-tableaux are shown in Figure \ref{fig:3-1}. The
nine maximal chains of $\tam_{4}$ are shown in Figure \ref{fig:3-2}.
We defined $\psi$ on maximal chains in \cite[Definition 3.1]{FN13}. 

\begin{figure}[h]
\begin{centering}
\includegraphics[scale=0.5]{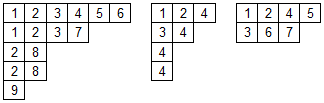}
\par\end{centering}

\caption{\label{fig:3-1}$\psi$-tableaux}
\end{figure}

\begin{figure}[h]
\begin{centering}
\includegraphics[scale=0.5]{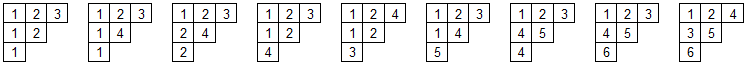}
\par\end{centering}

\caption{\label{fig:3-2}The nine maximal chains of $\tam_{4}$ via $\psi$}
\end{figure}

\begin{prop}
\label{prop:PsiIsInjective}$\psi$ is injective.\end{prop}
\begin{proof}
Let $T$ be a $\psi$-tableau. Then $T=\psi(C)$ for some $C=(\emptyset=Y_{0}\gtrdot Y_{1}\gtrdot\cdots\gtrdot Y_{l(T)})$.
The length of $T$ is the length of $C$. For each $r\in\{0,1,\ldots,l(T)\}$,
$Y_{r}=\{(x,y)\in T^{(r)}\}$.
\end{proof}
Because $\psi$ is injective, for each $n\geq1$ and for each Young
diagram $Y\in\tam_{n}$, $\psi$ extends to a bijection of sets between
saturated chains in $\tam_{n}$ of length $l$ whose minimal element
is $Y$ and maximal element is $\emptyset$, and $\psi$-tableaux
of length $l$ and the shape of $Y$. Since we identify vertices in
$\tam_{n}$ as Young diagrams contained in $\delta_{n-1}$, $\psi$
induces the following examples of bijective correspondences between:
\begin{itemize}
\item elements of $\chain_{i}(n)$ and $\psi$-tableaux of length $n+i$
and shape $\delta_{n-1}$,
\item maximal chains in $\tam_{n}$ and $\psi$-tableaux of shape $\delta_{n-1}$,
and
\item saturated chains in $\tam_{n}$ whose maximal element is $\emptyset$
and $\psi$-tableaux contained in $\delta_{n-1}$.\end{itemize}
\begin{defn}
\label{def:RSetBeginsEndsOuterDiagonal}A nonempty subset of boxes
of a Young diagram \textit{\textcolor{green}{begins}} and \textit{\textcolor{green}{ends}}
in its rows of minimum and maximum index, respectively. Similarly,
we may refer to the \textit{\textcolor{green}{begin-box}} or \textit{\textcolor{green}{end-box}}
of that subset.

If $Y\neq\emptyset$, then the \textit{\textcolor{green}{outer diagonal}}
of $Y$ is the set of boxes $\{(x,y)\in Y\mid x+y=m\}$ where $m$
is the maximum of $\{x+y\mid(x,y)\in Y\}$; otherwise, the outer diagonal
of $Y=\emptyset$ is the empty set. 
\end{defn}
The outer diagonal of a $\psi$-tableau of shape $\delta_{n-1}$,
for $n\geq1$, is the set of boxes $\{(k,n-k)\mid k\in[n-1]\}$.

Next we characterize $\psi$-tableau. Statements \ref{prop:PsiTabChar}(\ref{enu:PsiTabChar2})-(\ref{enu:PsiTabChar3}),
listed for convenience, follow from \ref{prop:PsiTabChar}(\ref{enu:PsiTabChar1}).
\begin{prop}
\label{prop:PsiTabChar}Characterization of $\psi$-tableaux. Let
$T$ be a tableau and $l=l(T)$. For each $k\in[l]$, let $B_{k}$
be the end-box of the $k$-set. Then:
\begin{enumerate}
\item \label{enu:PsiTabChar1}$T$ is a $\psi$-tableau if and only if for
each $k\in[l]$, the $k$-set is the $B_{k}$-strip in $T^{(k)}$. 
\item \label{enu:PsiTabChar2}Fix $r\in\{0,1,\ldots,l\}$. $T$ is a $\psi$-tableau
if and only if $T^{(r)}$ is a $\psi$-tableau, and for each $k\in\{r+1,r+2,\ldots,l\}$,
the $k$-set is the $B_{k}$-strip in $T^{(k)}$.
\item \label{enu:PsiTabChar3}If $l>0$, then $T$ is a $\psi$-tableau
if and only if $T^{(l-1)}$ is a $\psi$-tableau, and the $l$-set
is the $B_{l}$-strip in $T$. 
\end{enumerate}
\end{prop}
\begin{proof}
(\ref{enu:PsiTabChar1}). For each $0\leq k\leq l$, let $Y_{k}=\{(x,y)\in T^{(k)}\}$.
Then $T$ is a $\psi$-tableau if and only if $\emptyset=Y_{0}\gtrdot Y_{1}\gtrdot\cdots\gtrdot Y_{l}$.
Each $Y_{k}$ has the shape of $T^{(k)}$, and $Y_{0}=\emptyset$.
Furthermore, for each $k\in[l]$, $B_{k}$ is a corner box in $Y_{k}$
and is the unique box of maximum row index of all the boxes removed
from $Y_{k}$ to obtain $Y_{k-1}$. Thus, by Proposition \ref{prop:CoveringYoungDiagrams},
$Y_{k-1}\gtrdot Y_{k}$ if and only if the set of boxes removed from
$Y_{k}$ to obtain $Y_{k-1}$ is the $B_{k}$-strip in $Y_{k}$ if
and only if the $k$-set is the $B_{k}$-strip in $T^{(k)}$.
\end{proof}
In each tableau of Figure \ref{fig:3-3}, the $(3,2)$-strip, $\{(1,3),(2,2),(3,2)\}$,
is grayed and the prime path of $(3,2)$ is bolded. In the first tableau
(of length $4$), the $4$-set, $\{(2,2),(3,2)\}$, does not agree
with the $(3,2)$-strip, so is not a $\psi$-tableau. The second tableau
is a $\psi$-tableau. 

\begin{figure}[h]
\begin{centering}
\includegraphics[scale=0.5]{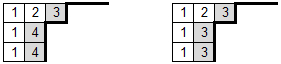}
\par\end{centering}

\caption{\label{fig:3-3}The first tableau is not a $\psi$-tableau; the second
one is.}
\end{figure}

The \textit{warm up} properties in the following proposition come
directly from basic definitions, the covering relation (Proposition
\ref{prop:CoveringYoungDiagrams}) and the definition of $\psi$ (Definition
\ref{def:Psi}). 
\begin{prop}
\label{prop:BasicProps}Basic Properties. \\
Let $Y$ be a Young diagram, and suppose $B\in Y$ is the last box
of the $x_{B}$-th row. Then:
\begin{enumerate}
\item \label{enu:BasicProps3}Each box in the outer diagonal is a corner
box.
\item \label{enu:BasicProps1}Let $h\geq1$. The height of the prime path
of $B$ is $h$ if and only if the $B$-strip is the set of last boxes
of all rows $j$ for which $x_{B}-h+1\leq j\leq x_{B}$.
\item Suppose $B'\in Y$ is the last box of the $x_{B'}$-th row. Then:

\begin{enumerate}
\item \label{enu:BasicProps4a}$B'$ is not in the $B$-strip if and only
if the $B$-strip is entirely above or entirely below row $x_{B'}$.
\item \label{enu:BasicProps4b}$B'$ is in the $B$-strip if and only if
the prime path of $B'$ is a subpath of the prime path of $B$.
\item \label{enu:BasicProps4c}Suppose that $B$ is in the outer diagonal
and that $x_{B'}>x_{B}$. Then the $B'$-strip begins in a row of
index greater than $x_{B}$.
\end{enumerate}
\end{enumerate}
Let $T$ be a $\psi$-tableau and $l=l(T)$. Then: 
\begin{enumerate}[start=4]
\item \label{enu:BasicProps5}Each box in the outer diagonal is the end-box
of a $j$-set, for some $j\in[l]$. Thus there are no repeat labels
in the outer diagonal. 
\item \label{enu:BasicProps6}For each $r\in\{0,1,\ldots,l\}$, $T^{(r)}$
is a $\psi$-tableau of length $r$. Also $\emptyset=T^{(0)}\gtrdot sh(T^{(1)})\gtrdot\cdots\gtrdot sh(T^{(l-1)})\gtrdot sh(T^{(l)})=sh(T)$.
\end{enumerate}
\end{prop}
The remaining material of this section is more technical and is necessary
to verify properties of the map $\phi_{i,n}^{r}$ defined in the next
section. 
\begin{prop}
\label{prop:DyckPathsDecreasingHeights}Let $C=(\hat{1}=P_{0}\gtrdot P_{1}\gtrdot\cdots\gtrdot P_{l})$
be a saturated chain in $\tam_{n}$ in terms of Dyck paths (of length
$2n$). Let $k\in[n]$. For each $0\leq j\leq l$, let $h_{j}$ be
the height of the prime Dyck subpath beginning with the $k$-th north
step of $P_{j}$. Then $h_{0}\geq h_{1}\geq\cdots\geq h_{l}$.\end{prop}
\begin{proof}
By induction on $l$, it suffices to show $h_{l-1}\geq h_{l}$ for
$l>0$. Let $Q$ be the prime Dyck subpath of $P_{l}$ beginning with
its $k$-th north step. Let $R$ be the prime Dyck subpath of $P_{l}$
that shifts to the left one unit in the transition $P_{l-1}\gtrdot P_{l}$
(see Proposition \ref{prop:CoveringDyckPaths}). We have a few cases
to check as outlined in Lemma \ref{lem:PrimePathsChar}. If $Q\cap R=\emptyset$
or $Q\subsetneq R$ or $R\subsetneq Q$ or $Q=R$, then $h_{l-1}=h_{l}$.
One case remains: $Q\cap R$ is a single point. If $R$ ends where
$Q$ begins, then $h_{l-1}=h_{l}$. If $Q$ ends where $R$ begins,
then the trailing east step of $Q$ swaps with $R$ in the transition
$P_{l-1}\gtrdot P_{l}$. In this case, $h_{l-1}>h_{l}$.
\end{proof}
Figure \ref{fig:3-4} is an example of a maximal chain in terms of
Dyck paths $P_{j}$ of length $8$. In the context of Proposition
\ref{prop:DyckPathsDecreasingHeights}, the sequences $(h_{0},h_{1},h_{2},h_{3},h_{4})$,
for $k\in\{1,2,3,4\}$, are $(4,4,2,1,1)$, $(3,1,1,1,1)$, $(2,2,2,2,1)$
and $(1,1,1,1,1)$, respectively. 

\begin{figure}[h]
\begin{centering}
$P_{0}$\hspace{0.5in}$P_{1}$\hspace{0.5in}$P_{2}$\hspace{0.5in}$P_{3}$\hspace{0.5in}$P_{4}$
\par\end{centering}

\begin{centering}
\includegraphics[scale=0.5]{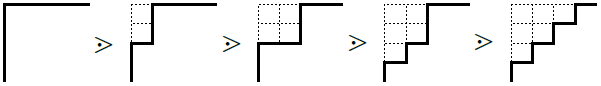}
\par\end{centering}

\caption{\label{fig:3-4}Example of Proposition \ref{prop:DyckPathsDecreasingHeights}}
\end{figure}

\begin{rem}
Knuth shows there are the Catalan number $C_{n}$ of \textit{forests}
on $n$ nodes, by creating a bijection between the sets of forests
and \textit{scope sequences} of length $n$ \cite{Knu06}. Scope sequences
also appear in \cite[Definition 9.1]{BjornerWachs97}. Proposition
\ref{prop:DyckPathsDecreasingHeights} follows from both references
directly from the covering relation in terms of scope sequences: Let
$P$ be a Dyck path of length $2n$. For each $1\leq k\leq n$, let
$H_{k}$ be the height of the prime Dyck subpath of $P$ beginning
with its $k$-th north step. Then the scope sequence of the forest
associated to $P$ is $(H_{1}-1,H_{2}-1,\ldots,H_{n}-1)$.
\end{rem}
Often times we will need to determine if a given tableau is a $\psi$-tableau.
The following definition and lemma will be utilized in this regard.
\begin{defn}
\label{def:EnclosureOfB-strip}We say that a set of boxes $S$ \textit{\textcolor{green}{translates}}
to a set of boxes $S'$ if each box in $S$ differs from one in $S'$
by the same horizontal and vertical amounts, and both sets are equal
in size, \textit{i.e.}, if for some constants $p$ and $q$, $S'=\{(x+p,y+q)\mid(x,y)\in S\}$.
Similarly, a subset of elements of a tableau may translate, or a path
of north and east steps may translate. 

Suppose $B=(x_{B},y_{B})$ is the last box of its row in a Young diagram
$Y$. Let $h$ be the height of the prime path of $B$. The \textit{\textcolor{green}{enclosure}}
of the $B$-strip in $Y$ is the set of boxes 
\[
\{(x,y)\in Y\mid x_{B}-h\leq x\leq x_{B}\text{ and }y_{B}\leq y\leq y_{B}+h\},
\]
where we consider row $0$ to be an infinite row of boxes. 
\end{defn}
In Figure \ref{fig:3-5}, the borders of $B$-strips and $B'$-strips
are bolded and their enclosures are grayed. $Y_{2}$ has shape $(3,2,1,1)$,
but the enclosure of its $B$-strip has shape $(5,3,2,1,1)$; it contains
$5$ boxes from row $0$.

\begin{figure}[h]
\begin{centering}
$Y_{1}$\hspace{0.75in}$Y_{1}'$\hspace{0.75in}$Y_{2}$\hspace{0.75in}$Y_{2}'$
\par\end{centering}

\begin{centering}
\includegraphics[scale=0.5]{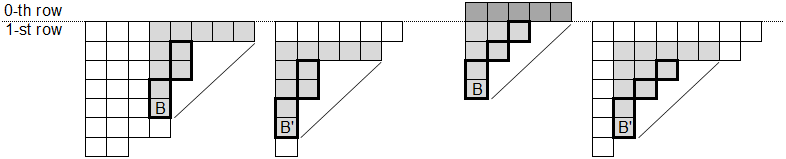}
\par\end{centering}

\caption{\label{fig:3-5}Examples of the Translation Lemma}
\end{figure}

In the context of Definition \ref{def:EnclosureOfB-strip}, notice
that the $B$-strip is a subset of its enclosure, and that the $B$-strip
begins in a row of index one more than where its enclosure begins.
The shape of the enclosure is the shape of a Young diagram. Furthermore,
the $B$-strip determines its enclosure and vice versa. 
\begin{lem}
\label{lem:TranslationLemma}Translation Lemma. Suppose $B$ and $B'$
are the last boxes of their rows in Young diagrams $Y$ and $Y'$,
respectively. Then the following conditions are equivalent.
\begin{itemize}
\item The enclosures of the $B$-strip in $Y$ and of the $B'$-strip in
$Y'$ have the same shape.
\item The enclosure of the $B$-strip in $Y$ translates to the enclosure
of the $B'$-strip in $Y'$.
\item The $B$-strip in $Y$ translates to the $B'$-strip in $Y'$.
\item The prime path of $B$ in $Y$ translates to the prime path of $B'$
in $Y'$. (The two prime paths are the same sequences of north and
east steps.)
\end{itemize}
If any of the conditions are satisfied, then the translations of the
various entities in the last three conditions are by the same horizontal
and vertical amounts as $B$ translates to $B'$.
\end{lem}
In Figure \ref{fig:3-5}, the $B$-strip in $Y_{1}$ translates $3$
units to the left and $1$ unit down to the $B'$-strip in $Y_{1}'$.
The $B$-strip in $Y_{2}$ translates $1$ unit to the right and $2$
units down to the $B'$-strip in $Y_{2}'$. In both cases, the four
conditions in the lemma may be verified. 

In what follows, we define $\alpha_{d}$ as a map on tableaux. Theorem
\ref{thrm:Alpha}(\ref{enu:Alpha5}) is a result specific to $\psi$-tableaux.
$\alpha_{d}$ is part of the map $\phi_{i,n}^{r}$ defined in Section
\ref{sec:phi}. 
\begin{thm}
\label{thrm:Alpha}Let $\mathbb{X}$ be the set of all tableaux and
$d\geq1$. Let $\mathbb{Y}_{d}$ be the set of all tableaux with row
$d$ identical to row $d+1$. Define \textcolor{green}{$\alpha_{d}(T)$}
to be the tableau obtained from $T\in\mathbb{X}$ by shifting any
rows of index greater than $d$ down one row and repeating row $d$
in row $d+1$. Then the map $T\mapsto\alpha_{d}(T)$ is a bijection
from $\mathbb{X}$ to $\mathbb{Y}_{d}$ (is clear). (Obtain $\alpha_{d}^{-1}(\widetilde{T})$
from $\widetilde{T}\in\mathbb{Y}_{d}$ by deleting row $d+1$ and
shifting any rows of index greater than $d+1$ up one row.) Let $T\in\mathbb{X}$
and $\widetilde{T}=\alpha_{d}(T)$ (equivalently, let $\widetilde{T}\in\mathbb{Y}_{d}$
and $T=\alpha_{d}^{-1}(\widetilde{T})$). Let $l=l(T)=l(\widetilde{T})$,
and $b$ the length of row $d$ in both $T$ and $\widetilde{T}$.
Then:
\begin{enumerate}
\item \label{enu:Alpha2}The first $d$ rows in $T$ are identical to those
rows in $\widetilde{T}$. In particular, $b=0$ if and only if $T=\widetilde{T}$.
\item \label{enu:Alpha3}The rows of index at least $d$ in $T$ are identical
to the rows of index at least $d+1$ in $\widetilde{T}$. Each of
the first $b$ columns of $\widetilde{T}$ has column length one more
than its respective column in $T$.
\item \label{enu:Alpha4}Suppose $b>0$. Let $B$ and $\widetilde{B}$ be
the end-boxes of the $l$-sets in $T$ and in $\widetilde{T}$, respectively,
and $x_{B}$ the row index of $B$. Then:

\begin{enumerate}
\item \label{enu:Alpha4a}$x_{B}<d$ if and only if $B=\widetilde{B}$ if
and only if the $l$-sets in $T$ and in $\widetilde{T}$ are equal. 
\item \label{enu:Alpha4b}$x_{B}\geq d$ if and only if $B$ translates
down one unit to $\widetilde{B}$.
\item \label{enu:Alpha4c}The $l$-set in $T$ begins in a row of index
greater than $d$ if and only if the $l$-set in $\widetilde{T}$
begins in a row of index greater than $d+1$ if and only if the $l$-set
in $T$ translates down one unit to the $l$-set in $\widetilde{T}$.
\item \label{enu:Alpha4d}Suppose $x_{B}\geq d$. Then the $l$-set in $T$
is the set of last boxes of all rows $j$ for $1\leq j\leq x_{B}$
if and only if the $l$-set in $\widetilde{T}$ is the set of last
boxes of all rows $j$ for $1\leq j\leq x_{B}+1$.
\end{enumerate}
\item \label{enu:Alpha5}Suppose the height of the prime path of row $d$
in $T$ is $d$ (equivalently, the height of the prime path of row
$d$ in $\widetilde{T}$ is $d$). Then $T$ is a $\psi$-tableau
if and only if $\widetilde{T}$ is a $\psi$-tableau.
\end{enumerate}
\end{thm}
$T_{1}$ and $T_{2}$ are both $\psi$-tableaux in Figure \ref{fig:3-6}.
For $d=5$, $T_{1}$ satisfies the assumptions of Theorem \ref{thrm:Alpha}(\ref{enu:Alpha5});
so $\alpha_{5}(T_{1})$ is a $\psi$-tableau (the prime paths of row
$5$ are bolded). For $d=3$, $T_{2}$ does not satisfy the assumptions
of Theorem \ref{thrm:Alpha}(\ref{enu:Alpha5}); the prime path of
row $3$ (in bold) has height $1$. In this case, $\alpha_{3}(T_{2})$
is not a $\psi$-tableau; the $(4,2)$-strip (grayed) does not equal
the $4$-set.

\begin{figure}[h]
\begin{centering}
$T_{1}$\hspace{0.85in}$\alpha_{5}(T_{1})$\hspace{0.7in}$T_{2}$\hspace{0.65in}$\alpha_{3}(T_{2})$\hspace{0.2in}
\par\end{centering}

\begin{centering}
\includegraphics[scale=0.5]{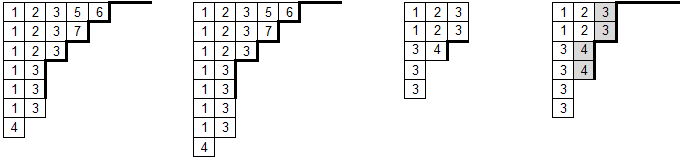}
\par\end{centering}

\caption{\label{fig:3-6}Examples of $\alpha_{d}$}
\end{figure}

\begin{proof}[Proof of Theorem \ref{thrm:Alpha}]
(\ref{enu:Alpha2})-(\ref{enu:Alpha4}) follow easily from the definition
of $\alpha_{d}$.

The proof of (\ref{enu:Alpha5}) is by induction on $l$. The base
case $l=0$ is handled in (\ref{enu:Alpha2}), so we may assume $b>0$
(and $l>0$). Assume the notation in (\ref{enu:Alpha4}). We claim
that 
\begin{equation}
\text{the }l\text{-set is the }B\text{-strip in }T\text{ if and only if the }l\text{-set is the }\widetilde{B}\text{-strip in }\widetilde{T}\text{.}\label{eq:ConditionA}
\end{equation}

This fact implies (\ref{enu:Alpha5}) as follows. Suppose $T$ is
a $\psi$-tableau for which the height of the prime path of row $d$
is $d$. Then $T^{(l-1)}$ is a $\psi$-tableau for which the height
of the prime path of row $d$ is $d$ (by Proposition \ref{prop:DyckPathsDecreasingHeights}).
By induction, it follows that $\widetilde{T}^{(l-1)}$ is a $\psi$-tableau,
and (\ref{eq:ConditionA}) implies that $\widetilde{T}$ is a $\psi$-tableau
(by Proposition \ref{prop:PsiTabChar}(\ref{enu:PsiTabChar3})). The
other direction is similar. 

To prove (\ref{eq:ConditionA}), first suppose $x_{B}<d$. Then the
conditions in (\ref{enu:Alpha4a}) hold. It follows from (\ref{enu:Alpha2}),
that the $B$-strip in $T$ equals the $\widetilde{B}$-strip in $\widetilde{T}$.
(\ref{eq:ConditionA}) follows.

Now suppose $x_{B}\geq d$. Let $P$ be the path along the contour
of $T$ between the bottom right corners of $B$ and $(d,b)$ (if
$B=(d,b)$, then $P$ is a point; otherwise, assume $P$ is made of
north and east steps from $B$ to $(d,b)$). Let $\widetilde{P}$
be the like path associated to $\widetilde{T}$ between $\widetilde{B}$
and $(d+1,b)$. It follows from (\ref{enu:Alpha3}) and (\ref{enu:Alpha4b})
that $\widetilde{B}$ has row index $x_{B}+1$ (and the same column
index as $B$) and that $P$ translates down one unit to $\widetilde{P}$.
There are two subcases.

In the first subcase, the prime path of $B$ in $T$ is a subpath
of $P$ (if and only if the prime path of $\widetilde{B}$ in $\widetilde{T}$
is a subpath of $\widetilde{P}$). Then the prime path of $B$ in
$T$ translates down one unit to the prime path of $\widetilde{B}$
in $\widetilde{T}$. (\ref{eq:ConditionA}) follows from the translation
lemma and (\ref{enu:Alpha4c}).

In the second subcase, the prime path of $(d,b)$ in $T$ is a subpath
of the prime path of $B$ (if and only if the prime path of $(d+1,b)$
in $\widetilde{T}$ is a subpath of the prime path of $\widetilde{B}$).
By assumption, the prime path of $(d,b)$ in $T$ has height $d$.
It follows that the prime path of $B$ in $T$ has height $x_{B}$.
Thus the $B$-strip in $T$ is the set of last boxes of all rows $j$
for $1\leq j\leq x_{B}$. The prime path of $(d,b)$ in $\widetilde{T}$
has height $d$, and is a subpath of the prime path of $(d+1,b$),
which in turn (we said) is a subpath of the prime path of $\widetilde{B}$.
It follows that the prime path of $\widetilde{B}$ in $\widetilde{T}$
has height $x_{B}+1$ ($\widetilde{B}$ has row index $x_{B}+1$).
Thus the $\widetilde{B}$-strip in $\widetilde{T}$ is the set of
last boxes of all rows $j$ for $1\leq j\leq x_{B}+1$. (\ref{eq:ConditionA})
then follows by (\ref{enu:Alpha4d}).
\end{proof}
Going forward, we identify chains by corresponding $\psi$-tableaux.

\section{\label{sec:phi}Plus-full-sets and a map on maximal chains}

The main focus of this section is to relate the map $\phi_{i,n}^{r}$,
and show that it is bijective; see Theorem \ref{thrm:PhiINR}. A maximal
chain in the image of $\psi$ may or may not possess a plus-full-set.
$\phi_{i,n}^{r}$ takes a maximal chain in $\chain_{i}(n)$ to one
in $\chain_{i}(n+1)$, where $r$ determines the domain and codomain
in terms of plus-full-sets. We first make basic definitions and then
build on $\alpha_{d}$ in Theorems \ref{thm:Beta} and \ref{thrm:PrePhi}.
We then specialize the map defined in Theorem \ref{thrm:PrePhi} by
modifying its domain and codomain to define $\phi_{i,n}^{r}$ in Theorem
\ref{thrm:PhiINR}. 

Recall that a maximal chain $C\in\chain_{i}(n)$ satisfies $sh(C)=\delta_{n-1}$
and $l(C)=n+i$. The outer diagonal of $C$ is the set of boxes $\{(k,n-k)\mid k\in[n-1]\}$.
The label in the box $(x,y)$ is denoted $C(x,y)$.
\begin{defn}
\label{def:FS}Let $C$ be a $\psi$-tableau of shape $\delta_{n-1}$,
for some $n\geq1$. For $r\in[l(C)]$, if the $r$-set begins in its
first row and ends in its outer diagonal, then we call the $r$-set
an \textit{\textcolor{green}{$r$-full-set}}, or more generally a
\textit{\textcolor{green}{full-set}}. In this case, there is a unique
box in the outer diagonal labeled with $r$ (by property \ref{prop:BasicProps}(\ref{enu:BasicProps5})),
\textit{i.e.}, there is a unique $k\in[n-1]$ satisfying $C(k,n-k)=r$.
The $r$-set is an \textit{\textcolor{green}{$r^{+}$-full-set}},
or more generally a \textit{\textcolor{green}{plus-full-set}}, if:
\begin{itemize}
\item The $r$-set is a full-set, and
\item its end-box $(k,n-k)$ satisfies $k=n-1$, or $k\in[n-2]$ and $C(k+1,n-k-1)<C(k,n-k)$
(the southwest neighbor of $(k,n-k)$ has a label less than $r$).
\end{itemize}
For each $r\in[n+i]$, \textit{\textcolor{green}{$\full_{i}^{r}(n)$}}
is the set of all $C\in\chain_{i}(n)$ satisfying:
\begin{itemize}
\item The $r$-set is a plus-full-set, and
\item for each $j\in[r-1]$, the $j$-set is not a plus-full-set.
\end{itemize}
\end{defn}
\begin{rem}
\label{rem:PFSSubsetsDisjoint}The $\full_{i}^{r}(n)$ are disjoint
subsets of $\chain_{i}(n)$. We denote a disjoint union of sets with
$\biguplus$.
\end{rem}
\begin{figure}[h]
\raggedleft{}%
\begin{minipage}[c]{0.3\columnwidth}%
\begin{center}
\includegraphics[scale=0.5]{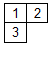}
\par\end{center}

\begin{center}
\caption{\label{fig:4-1}$\chain_{0}(3)$}

\par\end{center}%
\end{minipage}\hfill{}%
\begin{minipage}[c]{0.6\columnwidth}%
\begin{center}
\includegraphics[scale=0.5]{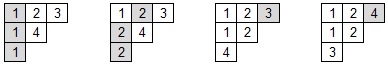}
\par\end{center}

\begin{center}
\caption{\label{fig:4-2}$\chain_{0}(4)=\biguplus_{r\in[4]}\full_{0}^{r}(4)$}

\par\end{center}%
\end{minipage}
\end{figure}

$\chain_{0}(3)$ consists of a single maximal chain, call it $C$,
shown in Figure \ref{fig:4-1}. Neither its $1$-set nor $3$-set
is a full-set, so neither is a plus-full-set. Its $2$-set is a full-set,
but is not a plus-full-set ($3=C(2,1)>C(1,2)=2$). Thus each of the
subsets $\full_{0}^{1}(3)$, $\full_{0}^{2}(3)$ and $\full_{0}^{3}(3)$
of $\chain_{0}(3)$ is the empty set.

Each of the subsets $\full_{0}^{1}(4)$, $\full_{0}^{2}(4)$, $\full_{0}^{3}(4)$
and $\full_{0}^{4}(4)$ of $\chain_{0}(4)$ consists of exactly one
maximal chain, listed in Figure \ref{fig:4-2}. The plus-full-set
that qualifies each maximal chain is grayed. $\chain_{0}(4)$ is the
disjoint union of these subsets. It is shown in Theorem \ref{thrm:GOrE2IPlus4}
that for all $n\geq2i+4$, $\chain_{i}(n)$ is the disjoint union
of nonempty subsets $\full_{i}^{r}(n)$ for which $r\in[3i+4]$.
\begin{thm}
\label{thm:Beta}Let $n\geq d\geq1$ and $r\geq0$. Define \textcolor{green}{$\beta_{d}(Y)$}
to be the tableau obtained from a tableau $Y$ of length $r$ by appending
a box labeled with $r+1$ to the end of all rows $j$ for which $1\leq j\leq d$.

Define \textcolor{green}{$\mathbb{X}_{d,n-1}^{r}$} to be the set
of all $\psi$-tableaux $X$ of length $r$ satisfying $(x1)$ $X$
is contained in $\delta_{n-1}$, $(x2)$ the length of row $d$ is
$n-d$, and $(x3)$ the prime path of row $d$ has height $d$.

Define \textcolor{green}{$\mathbb{Z}_{d,n}^{r+1}$} to be the set
of all $\psi$-tableaux $Z$ of length $r+1$ satisfying $(z1)$ $Z$
is contained in $\delta_{n}$, $(z2)$ rows $d$ and $d+1$ of $Z^{(r)}$
are identical, $(z3)$ the end-box of the $(r+1)$-set is $(d,n-d+1)$,
and $(z4)$ the prime path of $(d,n-d+1)$ has height $d$. 
\begin{enumerate}
\item \label{enu:ThmBetaLengths}Suppose $X$ is a tableau of length $r$
which satisfies $(x2)$, and let $Z=\beta_{d}(\alpha_{d}(X))$. Then
each of the first $d$ rows in $Z$ has row length one more than its
respective row in $X$, and each of the first $n-d$ columns in $Z$
has column length one more than its respective column in $X$. 
\item \label{enu:ThmBetaBiject}The map $X\mapsto\beta_{d}(\alpha_{d}(X))$
is a bijection from $\mathbb{X}_{d,n-1}^{r}$ to $\mathbb{Z}_{d,n}^{r+1}$
(see Figure \ref{fig:4-3}).
\end{enumerate}
\end{thm}
\begin{figure}[h]
\begin{centering}
{\footnotesize{}\hspace{0.28in}$X\in\mathbb{X}_{5,4}^{4}$\hspace{0.35in}$\alpha_{5}(X)$\hspace{0.26in}$Z=\beta_{5}(\alpha_{5}(X))\in\mathbb{Z}_{5,5}^{5}$\hspace{0.34in}$X\in\mathbb{X}_{3,5}^{5}$\hspace{0.45in}$\alpha_{3}(X)$\hspace{0.37in}$Z=\beta_{3}(\alpha_{3}(X))\in\mathbb{Z}_{3,6}^{6}$}
\par\end{centering}{\footnotesize \par}

\begin{centering}
\includegraphics[scale=0.5]{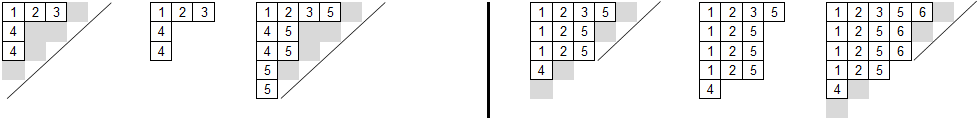}
\par\end{centering}

\caption{\label{fig:4-3}Examples of $Z=\beta_{d}(\alpha_{d}(X))$ for $X\in\mathbb{X}_{d,n-1}^{r}$:
$d=n$ in the first example, and $d\in[n-1]$ in the second. The grayed
boxes complete $\widetilde{X}$ (of shape $\delta_{n-1}$) and $\widetilde{Z}$
(of shape $\delta_{n}$) from $X$ and $Z$, respectively (see Corollary
\ref{cor:Beta}(\ref{enu:CorBetaUnlabeledTranslate})-(\ref{enu:CorBetaAllUnlabeled})).}
\end{figure}

\begin{proof}
(\ref{enu:ThmBetaLengths}) follows immediately by Theorem \ref{thrm:Alpha}(\ref{enu:Alpha2})-(\ref{enu:Alpha3})
and the definition of $\beta_{d}$.

(\ref{enu:ThmBetaBiject}). To prove that this map is well-defined,
suppose $X\in\mathbb{X}_{d,n-1}^{r}$. Let $Y=\alpha_{d}(X)$ and
$Z=\beta_{d}(Y)$. $Y$ and $Z$ have lengths $r$ and $r+1$, respectively.
Since $X$ satisfies $(x1)$, it follows from (\ref{enu:ThmBetaLengths})
that $Z$ satisfies $(z1)$. $Y$ satisfies $(x2)$ and its $d$-th
and $(d+1)$-th rows are identical; thus $Z$ satisfies $(z2)$ and
$(z3)$. $Y$ is a $\psi$-tableau by Theorem \ref{thrm:Alpha}(\ref{enu:Alpha5}).
The prime path of row $d$ in $Y$ has height $d$ and translates
one unit to the right to the prime path of $(d,n-d+1)$ in $Z$ (as
implied by the labeling of the $(r+1)$-set); thus $Z$ satisfies
$(z4)$, and the $(r+1)$-set is the $(d,n-d+1)$-strip in $Z$. $Z$
is a $\psi$-tableau by Proposition \ref{prop:PsiTabChar}(\ref{enu:PsiTabChar3}),
and thus $Z\in\mathbb{Z}_{d,n}^{r+1}$. From the fact that $X=\alpha_{d}^{-1}(Z^{(r)})$,
this map is injective.

To prove that this map is surjective, suppose $Z\in\mathbb{Z}_{d,n}^{r+1}$.
Because $Z$ satisfies $(z2)$, we can let $X=\alpha_{d}^{-1}(Z^{(r)})$.
Evidently $\beta_{d}(\alpha_{d}(X))=\beta_{d}(Z^{(r)})=Z$: We only
must show that $X\in\mathbb{X}_{d,n-1}^{r}$. Since $Z$ satisfies
$(z3)$, we have that $Z^{(r)}$ satisfies $(x2)$, and thus also
$X$ satisfies $(x2)$. Since $Z$ satisfies $(z1)$, it follows from
(\ref{enu:ThmBetaLengths}) that $X$ satisfies $(x1)$. Since $Z$
satisfies $(z4)$, it follows from Proposition \ref{prop:DyckPathsDecreasingHeights}
that $Z^{(r)}$ satisfies $(x3)$. $X$ satisfies $(x3)$ and is a
$\psi$-tableau by Theorem \ref{thrm:Alpha}(\ref{enu:Alpha5}). Thus
$X\in\mathbb{X}_{d,n-1}^{r}$.\end{proof}
\begin{cor}
\label{cor:Beta}Let $n\geq d\geq1$ and $r\geq0$. Suppose $X\in\mathbb{X}_{d,n-1}^{r}$,
and let $Z=\beta_{d}(\alpha_{d}(X))$ (equivalently, suppose $Z\in\mathbb{Z}_{d,n}^{r+1}$,
and let $X=\alpha_{d}^{-1}(\beta_{d}^{-1}(Z))$). Then:
\begin{enumerate}
\item \label{enu:CorBetaODBoxes}If $d=n$, then row $d$ in $X$ is empty,
and $(d,n-d+1)=(n,1)\in Z$ is in the outer diagonal of $\delta_{n}$.
If $d\in[n-1]$, then $(d,n-d)\in X$ is in the outer diagonal of
$\delta_{n-1}$, and $(d+1,n-d),(d,n-d+1)\in Z$ are in the outer
diagonal of $\delta_{n}$.
\end{enumerate}
Overlay $X$ on top of the Young diagram of shape $\delta_{n-1}$
so that $X$ is positioned to the top left, and call this construction
$\widetilde{X}$. Overlay $Z$ on top of $\delta_{n}$ in the like
manner to obtain $\widetilde{Z}$ (see Figure \ref{fig:4-3}).
\begin{enumerate}[start=2]
\item \label{enu:CorBetaUnlabeledTranslate}Any unlabeled boxes in $\widetilde{X}$
of row index less than $d$ (equivalently, to the right of column
$n-d$) translate to the right one unit to (and have the same skew
shape as) any unlabeled boxes in $\widetilde{Z}$ of row index less
than $d$ (equivalently, to the right of column $n-d+1$). Any unlabeled
boxes in $\widetilde{X}$ to the left of column $n-d$ (equivalently,
of row index greater than $d$) translate down one unit to (and have
the same skew shape as) any unlabeled boxes in $\widetilde{Z}$ to
the left of column $n-d$ (equivalently, of row index greater than
$d+1$).
\item \label{enu:CorBetaAllUnlabeled}Any unlabeled boxes in $\widetilde{X}$
and $\widetilde{Z}$ are accounted for in (\ref{enu:CorBetaUnlabeledTranslate}).
\end{enumerate}
\end{cor}
\begin{proof}
(\ref{enu:CorBetaODBoxes}) is clear. (\ref{enu:CorBetaUnlabeledTranslate})
and (\ref{enu:CorBetaAllUnlabeled}) follow from (\ref{enu:CorBetaODBoxes})
and Theorem \ref{thm:Beta}(\ref{enu:ThmBetaLengths}).\end{proof}
\begin{lem}
\label{lem:PsiTableauxEqualLengthRowsProp}Let $d\geq1$, and suppose
$T$ is a $\psi$-tableau such that rows $d$ and $d+1$ have equal
lengths. Then those rows are identical.\end{lem}
\begin{proof}
Let $b$ be the length of row $d$. If $b=0$, then this is clear,
so assume $b>0$. By induction it suffices to show that $T(d,b)=T(d+1,b)$.
Let $r=T(d+1,b)$ and $B$ the end-box of the $r$-set. $T^{(r)}$
contains $(d+1,b)$ so must contain $(d,b)$. The prime path of $(d,b)$
in $T^{(r)}$ is a subpath of the prime path of $(d+1,b)$ which in
turn is a subpath of the prime path of $B$. By property \ref{prop:BasicProps}(\ref{enu:BasicProps4b}),
$T(d+1,b)=T(d,b)$. \end{proof}
\begin{thm}
\label{thrm:PrePhi}Let $i\geq-1$, $n\geq1$ and $0\leq r\leq n+i$.
For $C\in\chain_{i}(n)$, define \textcolor{green}{$\widehat{C}(r)$}
as follows. If there exists $k\in[n-1]$ such that $C(k,n-k)\leq r$,
then let $d$ be minimal for $k$; otherwise, set $d=n$. $\widehat{C}(r)$
is the tableau obtained after performing the following iterative steps:
\begin{enumerate}
\item \label{enu:PrePhi1}Start with $\alpha_{d}(C^{(r)})$.
\item \label{enu:PrePhi2}Obtain $\beta_{d}(\alpha_{d}(C^{(r)}))$.
\item \label{enu:PrePhi3}For elements greater than $r$ in $C$:

\begin{enumerate}
\item \label{enu:PrePhi3a}For all those of row index less than $d$, translate
them to the right one unit to our construction, and increment their
labels by one.
\item \label{enu:PrePhi3b}For all those of row index greater than $d$,
translate them down one unit to our construction, and increment their
labels by one.
\end{enumerate}
\end{enumerate}
Then the map $C\mapsto\widehat{C}(r)$ is a bijection from $\chain_{i}(n)$
to $\{\widetilde{C}\in\chain_{i}(n+1)\mid\widetilde{C}$ has an $(r+1)^{+}$-full-set$\}$.
\end{thm}
\begin{center}
\begin{figure}[h]
\begin{minipage}[c]{0.14\columnwidth}%
\vspace{0.15in}

\begin{flushright}
$C$
\par\end{flushright}

\vspace{0.33in}

\begin{flushright}
\textit{Step 1}\\
\textit{$\alpha_{d}(C^{(r)})$}
\par\end{flushright}

\vspace{0.27in}

\begin{flushright}
\textit{Step 2}\\
\textit{$\beta_{d}(\alpha_{d}(C^{(r)}))$}
\par\end{flushright}

\vspace{0.25in}

\begin{flushright}
\textit{Step }3\\
$\widehat{C}(r)$
\par\end{flushright}%
\end{minipage}\hspace{5bp}%
\begin{minipage}[c]{0.81\columnwidth}%
\begin{flushleft}
\hspace{0.1in}$\begin{array}{c}
r=0\\
d=5
\end{array}$\hspace{0.2in}$\begin{array}{c}
r=1\\
d=5
\end{array}$\hspace{0.2in}$\begin{array}{c}
r=2\\
d=5
\end{array}$\hspace{0.2in}$\begin{array}{c}
r=3\\
d=3
\end{array}$\hspace{0.2in}$\begin{array}{c}
r=4\\
d=3
\end{array}$\hspace{0.2in}$\begin{array}{c}
r=5\\
d=1
\end{array}$\hspace{0.2in}$\begin{array}{c}
r=6\\
d=1
\end{array}$
\par\end{flushleft}

\begin{flushleft}
\includegraphics[scale=0.55]{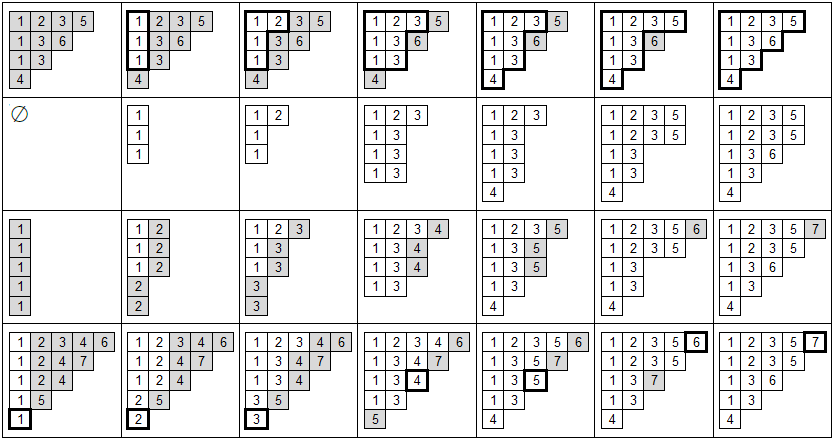}
\par\end{flushleft}%
\end{minipage}\caption{\label{fig:4-4}$\widehat{C}(r)$ is computed for each $0\leq r\leq6$
for a $C\in\chain_{1}(5)$.}
\end{figure}

\par\end{center}

Examples for a maximal chain $C\in\chain_{1}(5)$ are shown in Figure
\ref{fig:4-4}. $\widehat{C}(r)$ is computed for each $0\leq r\leq6$.
The outline of $C^{(r)}$ is bolded in the first row of the figure,
and its image under $\alpha_{d}$ is shown in \textit{Step 1}. Boxes
labeled with $r+1$ resulting from $\beta_{d}$ are grayed in \textit{Step
2}. Elements greater than $r$ in $C$ are grayed in the first row
of the figure. Their translated counterparts are grayed in \textit{Step
3}, and the border of $(d,n-d+1)$ is bolded. 
\begin{proof}[Proof of Theorem \ref{thrm:PrePhi}]
Suppose $C\in\chain_{i}(n)$. Our choice of $d$ is unique for $C$
and $r$. If $d=n$, then row $d$ in $C$ is empty, and $C^{(r)}$
is contained in $\delta_{n-2}$. If $d\in[n-1]$, then $C(d,n-d)\leq r$,
and $(d,n-d)$ is the box of least row index in the outer diagonal
of $C^{(r)}$. In either case, the height of the prime path of row
$d$ in $C^{(r)}$ is $d$. Thus $C^{(r)}\in\mathbb{X}_{d,n-1}^{r}$,
and evidently $\beta_{d}(\alpha_{d}(C^{(r)}))\in\mathbb{Z}_{d,n}^{r+1}$.
It also follows from Corollary \ref{cor:Beta}(\ref{enu:CorBetaUnlabeledTranslate})-(\ref{enu:CorBetaAllUnlabeled})
and by way of steps (\ref{enu:PrePhi3a})-(\ref{enu:PrePhi3b}) that
$\widehat{C}(r)$ is a tableau of shape $\delta_{n}$ and of length
$l(C)+1=n+i+1$. Moreover, the $(r+1)$-set in $\widehat{C}(r)$ begins
in the first row, and its end-box $(d,n-d+1)$ is in the outer diagonal
(see Corollary \ref{cor:Beta}(\ref{enu:CorBetaODBoxes})); so it
is a full-set. If $d\in[n-1]$, then $(d+1,n-d)\in\alpha_{d}(C^{(r)})=(\widehat{C}(r))^{(r)}$
which implies that $(\widehat{C}(r))(d+1,n-d)\leq r<r+1=(\widehat{C}(r))(d,n-d+1)$.
Thus the $(r+1)$-set is a plus-full-set. 

On the other hand suppose $\widetilde{C}\in\chain_{i}(n+1)$ has an
$(r+1)^{+}$-full-set. The $(r+1)^{+}$-full-set in $\widetilde{C}$
begins in the first row and there is a unique $d\in[n]$ satisfying
$\widetilde{C}(d,n-d+1)=r+1$. The height of the prime path of $(d,n-d+1)$
in $\widetilde{C}^{(r+1)}$ is $d$. We claim that rows $d$ and $d+1$
in $\widetilde{C}^{(r)}$ are identical. By Lemma \ref{lem:PsiTableauxEqualLengthRowsProp},
it suffices to show that those rows have equal length. If $d=n$,
then $\widetilde{C}(d,n-d+1)=\widetilde{C}(n,1)=r+1$, which implies
that rows $d$ and $d+1$ are empty in $\widetilde{C}^{(r)}$. So
suppose $d\in[n-1]$. We have that $\widetilde{C}(d+1,n-d)<\widetilde{C}(d,n-d+1)=r+1$
(by Definition \ref{def:FS}) and that $\widetilde{C}(d,n-d)<\widetilde{C}(d,n-d+1)=r+1$
(rows strictly increase). Thus $(d+1,n-d)$ and $(d,n-d)$ are the
last boxes of their rows in $\widetilde{C}^{(r)}$. Our claim is proved.
Thus $\widetilde{C}^{(r+1)}\in\mathbb{Z}_{d,n}^{r+1}$, and $\alpha_{d}^{-1}\beta_{d}^{-1}(\widetilde{C}^{(r+1)})\in\mathbb{X}_{d,n-1}^{r}$.
It also follows from Corollary \ref{cor:Beta}(\ref{enu:CorBetaUnlabeledTranslate})-(\ref{enu:CorBetaAllUnlabeled})
that the construction obtained from $\widetilde{C}$ by reversing
steps (\ref{enu:PrePhi1})-(\ref{enu:PrePhi3}) is a tableau of shape
$\delta_{n-1}$ and of length $l(\widetilde{C})-1=n+i$.

If $r=n+i$, we are done; otherwise, assume $0\leq r<n+i$. Suppose
that $\widetilde{C}$ is obtained from $C\in\chain_{i}(n)$ by way
of steps (\ref{enu:PrePhi1})-(\ref{enu:PrePhi3}), or that $C$ is
obtained from $\widetilde{C}\in\chain_{i}(n+1)$ having an $(r+1)^{+}$-full-set
by reversing those steps. We will show that $C$ is a $\psi$-tableau
if and only if $\widetilde{C}$ is a $\psi$-tableau. Based on this
fact, the map $C\mapsto\widehat{C}(r)$ is well-defined and surjective,
and it is clearly injective. Let $r<k\leq n+i$, $B_{k}$ the end-box
of the $k$-set in $C$, and $\widetilde{B}_{k+1}$ the end-box of
the $(k+1)$-set in $\widetilde{C}$. We claim that
\begin{equation}
\text{the }k\text{-set is the }B_{k}\text{-strip in }C{}^{(k)}\text{ if and only if the }(k+1)\text{-set is the }\widetilde{B}_{k+1}\text{-strip in }\widetilde{C}^{(k+1)}.\label{eq:ConditionX}
\end{equation}
Since $C^{(r)}$ and $\widetilde{C}^{(r+1)}$ are $\psi$-tableaux,
(\ref{eq:ConditionX}) implies that $C$ is a $\psi$-tableau if and
only if $\widetilde{C}$ is a $\psi$-tableau by way of Proposition
\ref{prop:PsiTabChar}(\ref{enu:PsiTabChar2}). The proof of (\ref{eq:ConditionX})
relies on Corollary \ref{cor:Beta}. There are two cases.

In the first case, the following three equivalent conditions hold
due to step (\ref{enu:PrePhi3a}) (or its reverse): $B_{k}\in C$
has row index less than $d$, $\widetilde{B}_{k+1}\in\widetilde{C}$
has row index less than $d$, and $B_{k}$ translates to the right
one unit to $\widetilde{B}_{k+1}$. Its clear that the $B_{k}$-strip
in $C^{(k)}$ translates to the right one unit to the $\widetilde{B}_{k+1}$-strip
in $\widetilde{C}^{(k+1)}$. (\ref{eq:ConditionX}) then follows by
way of step (\ref{enu:PrePhi3a}).

In the second case, the following three equivalent conditions hold
due to step (\ref{enu:PrePhi3b}) (or its reverse): $B_{k}\in C$
has row index greater than $d$, $\widetilde{B}_{k+1}\in\widetilde{C}$
has row index greater than $d+1$, and $B_{k}$ translates down one
unit to $\widetilde{B}_{k+1}$. Since $(d,n-d)$ is in the outer diagonal
of $C^{(k)}$, it follows by property \ref{prop:BasicProps}(\ref{enu:BasicProps4c})
that the $B_{k}$-strip in $C^{(k)}$ begins in a row of index greater
than $d$. Likewise, since $(d+1,n-d)$ is in the outer diagonal of
$\widetilde{C}^{(k+1)}$, the $\widetilde{B}_{k+1}$-strip in $\widetilde{C}^{(k+1)}$
begins in a row of index greater than $d+1$. The shape of $\{(x,y)\in C^{(k)}\mid x\geq d\}$
is the shape of $\{(x,y)\in\widetilde{C}^{(k+1)}\mid x\geq d+1\}$
and those sets contain the enclosures of the $B_{k}$-strip in $C^{(k)}$
and of the $\widetilde{B}_{k+1}$-strip in $\widetilde{C}^{(k+1)}$,
respectively. By the translation lemma, the $B_{k}$-strip in $C^{(k)}$
translates down one unit to the $\widetilde{B}_{k+1}$-strip in $\widetilde{C}^{(k+1)}$.
(\ref{eq:ConditionX}) then follows by way of step (\ref{enu:PrePhi3b}).
\end{proof}
The map in Theorem \ref{thrm:PrePhi} preserves full-sets and plus-full-sets
in the following sense.
\begin{prop}
\label{prop:PFSCondition}Let $i\geq-1$, $n\geq1$ and $0\leq r\leq n+i$.
Suppose $C\in\chain_{i}(n)$. Then:
\begin{enumerate}
\item \label{enu:PFSCondition1}For $j\in[r]$, $C$ has a $j$-full-set
(respectively, $j^{+}$-full-set) if and only if $\widehat{C}(r)$
has a $j$-full-set (respectively, $j^{+}$-full-set).
\item \label{enu:PFSCondition2}For $r<j\leq n+i$, $C$ has a $j$-full-set
(respectively, $j^{+}$-full-set) if and only if $\widehat{C}(r)$
has a $(j+1)$-full-set (respectively, $(j+1)^{+}$-full-set).
\end{enumerate}
\end{prop}
\begin{proof}
By our choice of $d$ in Theorem \ref{thrm:PrePhi}, $d\in[n]$ and
the following two items hold.
\begin{itemize}
\item Each box of row index less than $d$ in the outer diagonal of $C$
has a label greater than $r$. If $d\in[n-1]$, then $C(d,n-d)\leq r$.
\item $(\widehat{C}(r))(d,n-d+1)=r+1$. Each box of row index less than
$d$ in the outer diagonal of $\widehat{C}(r)$ has a label greater
than $r+1$. If $d\in[n-1]$, then $(\widehat{C}(r))(d+1,n-d)\leq r$.
\end{itemize}
(\ref{enu:PFSCondition1}). Suppose $j\in[r]$. If the $j$-set is
a full-set in $C$, it ends in a row of index at least $d$. If the
$j$-set is a full-set in $\widehat{C}(r)$, it ends in a row of index
at least $d+1$. It then follows from step \ref{thrm:PrePhi}(\ref{enu:PrePhi1})
that $C$ has a $j$-full-set if and only if $\widehat{C}(r)$ has
a $j$-full-set. In that case, we have $C(k,n-k)=j=(\widehat{C}(r))(k+1,n-k)$
for a unique $k$ satisfying $d\leq k\leq n-1$. Assume that case.
If $k=n-1$, then both $C$ and $\widehat{C}(r)$ have $j^{+}$-full-sets,
so suppose $k\in[n-2]$. Then it also follows from step \ref{thrm:PrePhi}(\ref{enu:PrePhi1})
that $C(k+1,n-k-1)<C(k,n-k)$ if and only if $(\widehat{C}(r))(k+2,n-k-1)<(\widehat{C}(r))(k+1,n-k)$
(in that case, $C(k+1,n-k-1)=(\widehat{C}(r))(k+2,n-k-1)$), finishing
the proof.

(\ref{enu:PFSCondition2}). Suppose $r<j\leq n+i$. If the $j$-set
is a full-set in $C$, it ends in a row of index less than $d$ (otherwise,
row $d$ in $C$ would have the label $j>r$). If the $(j+1)$-set
is a full-set in $\widehat{C}(r)$, it ends in a row of index less
than $d$ (otherwise, row $d$ in $\widehat{C}(r)$ would have the
label $j+1>r+1$). It then follows from step \ref{thrm:PrePhi}(\ref{enu:PrePhi3a})
that $C$ has a $j$-full-set if and only if $\widehat{C}(r)$ has
a $(j+1)$-full-set. In that case, we have $C(k,n-k)=j$ and $(\widehat{C}(r))(k,n-k+1)=j+1$
for a unique $k$ satisfying $1\leq k<d$. Supposing that case, there
are three subcases.

In the first subcase, suppose that $k=n-1,$ so $d=n$. Thus $C(n-1,1)=j$,
and $(\widehat{C}(r))(n,1)=r+1<j+1=(\widehat{C}(r))(n-1,2)$, so that
$C$ has a $j^{+}$-full-set and $\widehat{C}(r)$ has a $(j+1)^{+}$-full-set.

In the second subcase, suppose that $k\in[n-2]$ and that $k+1=d$.
Then $C(k+1,n-k-1)=C(d,n-d)\leq r<j=C(k,n-k)$ and $(\widehat{C}(r))(k+1,n-k)=(\widehat{C}(r))(d,n-d+1)=r+1<j+1=(\widehat{C}(r))(k,n-k+1)$,
so that $C$ has a $j^{+}$-full-set and $\widehat{C}(r)$ has a $(j+1)^{+}$-full-set.

In the third subcase, suppose that $k\in[n-2]$ and that $k+1<d$.
Then $C(k+1,n-k-1)>r$ and $(\widehat{C}(r))(k+1,n-k)>r+1$. It follows
from step \ref{thrm:PrePhi}(\ref{enu:PrePhi3a}) that $C(k+1,n-k-1)+1=(\widehat{C}(r))(k+1,n-k)$,
and thus that $C(k+1,n-k-1)<C(k,n-k)=j$ if and only if $(\widehat{C}(r))(k+1,n-k)<(\widehat{C}(r))(k,n-k+1)=j+1$.
Therefore $C$ has a $j^{+}$-full-set if and only if $\widehat{C}(r)$
has a $(j+1)^{+}$-full-set.\end{proof}
\begin{thm}
\label{thrm:PhiINR}Let $i\geq-1$, $n\geq1$ and $0\leq r\leq n+i$.
The map
\begin{eqnarray*}
\phi_{i,n}^{r}:\{C\in\chain_{i}(n)\mid\forall j\in[r],\,C\notin\full_{i}^{j}(n)\} & \rightarrow & \full_{i}^{r+1}(n+1)\\
C & \mapsto & \widehat{C}(r)
\end{eqnarray*}
is a bijection.\end{thm}
\begin{rem}
\label{rem:DomainPhiEquivCond}The condition on the domain of $\phi_{i,n}^{r}$,
$\forall j\in[r]$, $C\notin\full_{i}^{j}(n)$, is equivalent to $\forall j\in[r]$,
the $j$-set in $C$ is not a plus-full-set.
\end{rem}
In the examples in Figure \ref{fig:4-4}, $C$ has no plus-full-sets.
Thus for each $0\leq r\leq6$, $C$ is in the domain of $\phi_{1,5}^{r}$.
\begin{proof}[Proof of Theorem \ref{thrm:PhiINR}]
In lieu of Theorem \ref{thrm:PrePhi}, this follows from Proposition
\ref{prop:PFSCondition}(\ref{enu:PFSCondition1}).\end{proof}
\begin{cor}
Let $i\geq-1$, $n\geq1$ and $0\leq r\leq n+i$. Then
\[
\#\full_{i}^{r+1}(n+1)=\#\chain_{i}(n)-\sum_{j=1}^{r}\#\full_{i}^{j}(n).
\]
\end{cor}
\begin{proof}
This is a direct implication of Theorem \ref{thrm:PhiINR}, recalling
Remark \ref{rem:PFSSubsetsDisjoint}.
\end{proof}

\section{\label{sec:formula}A formula for the number of maximal chains of
length $n+i$ in $\tam_{n}$}

The technical work in verifying the bijectivity of $\phi_{i,n}^{r}$
is complete! In this section, we gather more on properties and consequences
of this map and tie our results together to write a recursive formula
for $\#\chain_{i}(n)$. $\phi_{i,n}^{r}$ maps a maximal chain to
one with one more plus-full-set (Proposition \ref{prop:NumPFS}).
We may then write each maximal chain having a plus-full-set uniquely
in terms of one with no plus-full-sets (Corollary \ref{cor:UniqueRepresentationByPhi}).
By relating this unique representation for a maximal chain to specific
plus-full-sets that it contains (Proposition \ref{prop:TtupleRelatesToPFS}),
we obtain an expression for $\#\chain_{i}(n)$; see equation (\ref{eq:=000023Ci(n)FirstExpr}).
For each $i\geq-1$, there exists a maximal chain in $\chain_{i}(2i+3)$
containing no plus-full-sets (Lemma \ref{lem:2iPlus3}), but surprisingly,
for all $n\geq2i+4$, each maximal chain in $\chain_{i}(n)$ has a
plus-full-set (Theorem \ref{thrm:GOrE2IPlus4}). We utilize these
latter two facts to refine our expression for $\#\chain_{i}(n)$ and
show that it is a polynomial of degree $3i+3$ in Theorem \ref{thm:MainThrm}.
\begin{prop}
\label{prop:NumPFS}Each $C$ in the domain of $\phi_{i,n}^{r}$ has
one less plus-full-set than its image.\end{prop}
\begin{proof}
Suppose $C$ is in the domain of $\phi_{i,n}^{r}$. By definition,
for each $j\in[r]$, neither $C$ nor $\phi_{i,n}^{r}(C)$ has a $j^{+}$-full-set.
Of course, $\phi_{i,n}^{r}(C)$ has the $(r+1)^{+}$-full-set. The
proof follows from Proposition \ref{prop:PFSCondition}(\ref{enu:PFSCondition2}).\end{proof}
\begin{defn}
\label{def:Nin}\textit{\textcolor{green}{$\nofull_{i}(n)$}} is the
set of all maximal chains in $\chain_{i}(n)$ having no plus-full-sets. 
\end{defn}
$\chain_{i}(n)$ is a disjoint union of the $n+i+1$ subsets, $\nofull_{i}(n)$
and $\full_{i}^{j}(n)$, $j\in[n+i]$, \textit{i.e.}, 
\begin{equation}
\chain_{i}(n)=\nofull_{i}(n)\biguplus\left(\biguplus_{j\in[n+i]}\full_{i}^{j}(n)\right).\label{eq:Ci(n)IsDisjUnion}
\end{equation}

\begin{cor}
\label{cor:UniqueRepresentationByPhi}Suppose the number of plus-full-sets
of some $C\in\chain_{i}(n)$ is $t>0$. Then there exists a unique
$\widetilde{C}_{1}\in\chain_{i}(n-1)$ and a unique $r_{1}$, such
that $C=\phi_{i,n-1}^{r_{1}}(\widetilde{C}_{1})$. This representation
may be extended to obtain unique representations
\begin{eqnarray*}
C & = & \phi_{i,n-1}^{r_{1}}(\widetilde{C}_{1})\\
 & = & \phi_{i,n-1}^{r_{1}}(\phi_{i,n-2}^{r_{2}}(\widetilde{C}_{2}))\\
 & \vdots\\
 & = & (\phi_{i,n-1}^{r_{1}}\circ\phi_{i,n-2}^{r_{2}}\circ\cdots\circ\phi_{i,n-t}^{r_{t}})(\widetilde{C}_{t}),
\end{eqnarray*}
until we arrive at $\widetilde{C}_{t}\in\nofull_{i}(n-t)$.\end{cor}
\begin{proof}
The codomain of $\phi_{i,n-1}^{r}$ is $\full_{i}^{r+1}(n)$. As $r$
varies, $0\leq r\leq n-1+i$, the $\full_{i}^{r+1}(n)$ are disjoint
subsets of $\chain_{i}(n)$. $C$ is an element of exactly one of
the $\full_{i}^{r+1}(n)$, so there exists a unique $r_{1}$, such
that $\phi_{i,n-1}^{r_{1}}$ has $C$ in its codomain. Since $\phi_{i,n-1}^{r_{1}}$
is bijective, there exists a unique $\widetilde{C}_{1}\in\chain_{i}(n-1)$,
such that $C=\phi_{i,n-1}^{r_{1}}(\widetilde{C}_{1})$. By Proposition
\ref{prop:NumPFS}, the number of plus-full-sets in $\widetilde{C}_{1}$
is $t-1$. If $t=1$, then $\widetilde{C}_{1}\in\nofull_{i}(n-1)$;
otherwise, $t-1>0$, and we may repeat until we arrive at $\widetilde{C}_{t}\in\nofull_{i}(n-t)\subseteq\chain_{i}(n-t)$.\end{proof}
\begin{rem}
\label{rem:NumPFSRange0ToNminusOne}A maximal chain in $\tam_{n}$
has at most $n-1$ plus-full-sets, as bounded by the $n-1$ strictly
increasing labels in its first row. 
\end{rem}
By Corollary \ref{cor:UniqueRepresentationByPhi}, the number of maximal
chains in $\chain_{i}(n)$ with exactly $t$ plus-full-sets, $1\leq t\leq n-1$,
is the number of representations
\begin{equation}
(\phi_{i,n-1}^{r_{1}}\circ\phi_{i,n-2}^{r_{2}}\circ\cdots\circ\phi_{i,n-t}^{r_{t}})(\widetilde{C})\label{eq:MaxChainRepresentation}
\end{equation}
over $t$-tuples $(r_{1},r_{2},\ldots,r_{t})$ and over $\widetilde{C}\in\nofull_{i}(n-t)$.
Each $t$-tuple $(r_{1},r_{2},\ldots,r_{t})$ must satisfy restrictions
imposed on the $r_{j}$ in Theorem \ref{thrm:PhiINR}:
\begin{prop}
\label{prop:TtupleRelatesToPFS}Suppose that $\widetilde{C}\in\nofull_{i}(n-t)$
for some $n$ and $t$, satisfying $1\leq t\leq n-1$. A $t$-tuple
$(r_{1},r_{2},\ldots,r_{t})$ for the representation (\ref{eq:MaxChainRepresentation}),
must only satisfy $0\leq r_{1}\leq r_{2}\leq\cdots\leq r_{t}\leq n-t+i$.
The number of these, hence the number of representations (\ref{eq:MaxChainRepresentation}),
is $\binom{n+i}{t}$. 

For a $t$-tuple $(r_{1},r_{2},\ldots,r_{t})$ which satisfies the
criteria, let $C=(\phi_{i,n-1}^{r_{1}}\circ\phi_{i,n-2}^{r_{2}}\circ\cdots\circ\phi_{i,n-t}^{r_{t}})(\widetilde{C})$.
The set of specific plus-full-sets in $C$ is 
\[
\{j\mid C\text{ has a }j^{+}\text{-full-set}\}=\{r_{1}+1,r_{2}+2,\ldots,r_{t}+t\},
\]
which is a $t$-element subset of $[n+i]$ unique to $(r_{1},r_{2},\ldots,r_{t})$.\end{prop}
\begin{proof}
Consider a $t$-tuple $(r_{1},r_{2},\ldots,r_{t})$ for the representation
(\ref{eq:MaxChainRepresentation}). By Theorem \ref{thrm:PhiINR},
since $\widetilde{C}\in\nofull_{i}(n-t)$, $r_{t}$ for $\phi_{i,n-t}^{r_{t}}$
must only satisfy $0\leq r_{t}\leq n-t+i$. We obtain $\phi_{i,n-t}^{r_{t}}(\widetilde{C})\in\full_{i}^{r_{t}+1}(n-t+1)$.
The $(r_{t}+1)$-set in $\phi_{i,n-t}^{r_{t}}(\widetilde{C})$ is
its only plus-full-set. By definition, $r_{t-1}$ for $\phi_{i,n-(t-1)}^{r_{t-1}}$
must only satisfy $0\leq r_{t-1}\leq r_{t}$. Continuing in this manner,
we find that $(r_{1},r_{2},\ldots,r_{t})$ must only satisfy $0\leq r_{1}\leq r_{2}\leq\cdots\leq r_{t}\leq n-t+i$.
The standard trick is to make the substitution $r_{k}=u_{k}-k$, obtaining
$0<u_{1}<u_{2}<\cdots<u_{t}\leq n+i$. The $t$-tuples $(u_{1},u_{2},\ldots,u_{t})$
which satisfy this are the $t$-element subsets of $[n+i]$. Moreover,
by \ref{prop:PFSCondition}(\ref{enu:PFSCondition2}), $\{u_{1},u_{2},\ldots,u_{t}\}=\{r_{1}+1,r_{2}+2,\ldots,r_{t}+t\}$
is the set $\{j\mid C$ has a $j^{+}$-full-set$\}$ for $C=\left(\phi_{i,n-1}^{r_{1}}\circ\phi_{i,n-2}^{r_{2}}\circ\cdots\circ\phi_{i,n-t}^{r_{t}}\right)(\widetilde{C})$.\end{proof}
\begin{cor}
\label{cor:EqualRepresentationOfPFS}There is equal representation
in $\chain_{i}(n)$ over equal size subsets of $[n+i]$ in terms of
specific plus-full-sets found in maximal chains: For each $t$-element
subset $U\subseteq[n+i]$, such that $0\leq t\leq n-1$, 
\begin{equation}
\#\{C\in\chain_{i}(n)\mid U=\{j\mid C\text{ has a }j^{+}\text{-full-set}\}\}=\#\nofull_{i}(n-t),\label{eq:EqualRepresExact}
\end{equation}
\begin{equation}
\#\{C\in\chain_{i}(n)\mid U\subseteq\{j\mid C\text{ has a }j^{+}\text{-full-set}\}\}=\#\chain_{i}(n-t).\label{eq:EqualRepresSubset}
\end{equation}
\end{cor}
\begin{proof}
The case $t=0$ is trivial so assume $1\leq t\leq n-1$. 

For equation (\ref{eq:EqualRepresExact}), let $S_{1}=\{C\in\chain_{i}(n)\mid U=\{j\mid C\text{ has a }j^{+}\text{-full-set}\}\}$.
Suppose $U=\{u_{1},u_{2},\ldots,u_{t}\}$, where $0<u_{1}<u_{2}<\cdots<u_{t}\leq n+i$.
By Corollary \ref{cor:UniqueRepresentationByPhi} and Proposition
\ref{prop:TtupleRelatesToPFS}, each $C\in S_{1}$ has the representation
\[
\left(\phi_{i,n-1}^{u_{1}-1}\circ\phi_{i,n-2}^{u_{2}-2}\circ\cdots\circ\phi_{i,n-t}^{u_{t}-t}\right)(\widetilde{C}),
\]
for a unique $\widetilde{C}\in\nofull_{i}(n-t)$, and this is an element
of $S_{1}$ for every $\widetilde{C}\in\nofull_{i}(n-t)$. 

For equation (\ref{eq:EqualRepresSubset}), let $S_{2}=\{C\in\chain_{i}(n)\mid U\subseteq\{j\mid C\text{ has a }j^{+}\text{-full-set}\}\}$.
Because of equation (\ref{eq:EqualRepresExact}), $\#S_{2}$ depends
only on $\#U=t$. It suffices to consider $U=[t]$. Suppose $C\in S_{2}$
has exactly $s\geq t$ plus-full-sets. Again by Corollary \ref{cor:UniqueRepresentationByPhi}
and Proposition \ref{prop:TtupleRelatesToPFS}, there exists a unique
$\widetilde{C}_{s}\in\nofull_{i}(n-s)$ and a unique $s$-tuple $(r_{1},r_{2},\ldots,r_{s})$,
such that 
\[
C=\left(\phi_{i,n-1}^{r_{1}}\circ\phi_{i,n-2}^{r_{2}}\circ\cdots\circ\phi_{i,n-s}^{r_{s}}\right)(\widetilde{C}_{s}),
\]
where $0\leq r_{1}\leq r_{2}\leq\cdots\leq r_{s}\leq n-s+i$ and $[t]\subseteq\{r_{1}+1,r_{2}+2,\ldots,r_{s}+s\}$.
For each $k\in[t]$, we must have $r_{k}=0$. Thus, each $C\in S_{2}$
has the representation
\[
\left(\phi_{i,n-1}^{0}\circ\phi_{i,n-2}^{0}\circ\cdots\circ\phi_{i,n-t}^{0}\right)(\widetilde{C}),
\]
for a unique $\widetilde{C}\in\chain_{i}(n-t)$, and this is an element
of $S_{2}$ for every $\widetilde{C}\in\chain_{i}(n-t)$.
\end{proof}
An expression for $\#\chain_{i}(n)$ is acquired from equation (\ref{eq:EqualRepresExact}).
(Note, the expression may be obtained directly from Proposition \ref{prop:TtupleRelatesToPFS}.)
\begin{equation}
\#\chain_{i}(n)=\sum_{t=0}^{n-1}\binom{n+i}{t}\#\nofull_{i}(n-t)=\sum_{t=1}^{n}\binom{n+i}{t+i}\#\nofull_{i}(t).\label{eq:=000023Ci(n)FirstExpr}
\end{equation}
The second expression follows from the first by reindexing and is
refined in Theorem \ref{thm:MainThrm}.

An expression for $\#\nofull_{i}(n)$ is obtained from equations (\ref{eq:EqualRepresExact}),
(\ref{eq:EqualRepresSubset}) and the principle of inclusion and exclusion.
\begin{equation}
\#\nofull_{i}(n)=\sum_{t=0}^{n-1}(-1)^{t}\binom{n+i}{t}\#\chain_{i}(n-t)=\sum_{t=1}^{n}(-1)^{n-t}\binom{n+i}{t+i}\#\chain_{i}(t).\label{eq:=000023Ni(n)FirstExpr}
\end{equation}
The second expression follows from the first by reindexing.

In Theorem \ref{thrm:GOrE2IPlus4}, the surprising fact is that for
all $n\geq2i+4$, every maximal chain in $\chain_{i}(n)$ has a plus-full-set.
For this condition on $\chain_{i}(n)$, we can do no better (the following
lemma). These facts enable us to reach our main objective in Theorem
\ref{thm:MainThrm}.
\begin{lem}
\label{lem:2iPlus3}For each $i\geq-1$, $\#\nofull_{i}(2i+3)>0$.\end{lem}
\begin{proof}
For $i=-1$, $\chain_{-1}(1)$ consists of only the null diagram,
so assume $i\geq0$. A qualifying maximal chain in $\tam_{n}$, $n=2i+3$,
will have the shape $\delta_{n-1}=(2i+2,\ldots,1)$ and length $n+i=3i+3$.
Let $C$ be a Young diagram of the desired shape. 

For all $0\leq k\leq i$, let $C(n-(2k+1),2k+1)=n+i-k$. Each $r$-set,
for $r\in[n,n+i]$, is labeled here and consists of a single box in
the outer diagonal. Specifically, $C(n-1,1)=n+i$ and $C(2,n-2)=n$,
and every other box between $(n-1,1)$ and $(2,n-2)$ gets a label.

For all $k\in[n-1]$, label the remaining unlabeled boxes in column
$k$ with $k$.

The resulting diagram is a $\psi$-tableau of the desired shape and
length (see Figure \ref{fig:5-1}). Moreover, the full-sets of $C$
end in the boxes $(n-(2k+2),2k+2)$, $0\leq k\leq i$, such that $C(n-(2k+1),2k+1)>C(n-(2k+2),2k+2)$,
so that each full-set is not a plus-full-set. 
\end{proof}
\begin{figure}[h]
\begin{centering}
\includegraphics[scale=0.5]{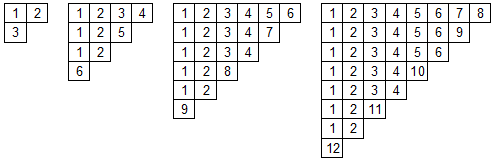}
\par\end{centering}

\caption{\label{fig:5-1}Examples of maximal chains in $\nofull_{i}(2i+3)$,
$i\in\{0,1,2,3\}$}
\end{figure}

\begin{thm}
\label{thrm:GOrE2IPlus4}For each $i\geq-1$ and for all $n\geq2i+4$,
\begin{enumerate}
\item \label{enu:GreaterOrEqual2iPlus41}$\#\chain_{i}(n)>0$,
\item \label{enu:GreaterOrEqual2iPlus42}every element of $\chain_{i}(n)$
has a plus-full-set, i.e., $\#\nofull_{i}(n)=0$,
\item \label{enu:GreaterOrEqual2iPlus43}$\chain_{i}(n)=\biguplus_{j\in[3i+4]}\full_{i}^{j}(n)$,
where each $\full_{i}^{j}(n)$ is nonempty.
\end{enumerate}
\end{thm}
\begin{proof}
(\ref{enu:GreaterOrEqual2iPlus41}). Since maximal chains in $\tam_{n}$
range in length from $n-1$ to $\binom{n}{2}$, it suffices to show
that $n-1\leq n+i\leq\binom{n}{2}$. Since $i\geq-1$, $n-1\leq n+i$.
Since $n\geq2i+4\geq2$, 
\[
n+i\leq\binom{n}{2}\Longleftrightarrow0\leq n^{2}-3n-2i,
\]
and
\[
n^{2}-3n-2i\geq n^{2}-3n+4-n=\left(n-2\right)^{2}\geq0.
\]

(\ref{enu:GreaterOrEqual2iPlus42}). Let $C\in\chain_{i}(n)$. Since
$n\geq2$, $C$ is not the null diagram. Suppose $n=2i+4+l$ for some
$l\geq0$. The length of $C$ is $n+i=3i+4+l$. There are no repeat
labels in the first row of $n-1=2i+3+l$ boxes. Likewise, there are
no repeat labels in the $n-1$ boxes in the outer diagonal. Let $x$
be the number of full-sets in $C$ and note that $(1,n-1)$ constitutes
a full set. The combined number of distinct labels in the first row
and outer diagonal is $x+2(n-1-x)\leq n+i$, thus 
\[
x\geq n-2-i=i+2+l.
\]
Suppose two full-sets end in boxes in adjacent columns, say in $(k,n-k)$
and $(k+1,n-k-1)$. Then $C(k+1,n-k-1)<C(k,n-k)$, and thus $C$ has
an $r^{+}$-full-set for $r=C(k,n-k)$. On the other hand suppose
no two full-sets end in adjacent columns. Then we require at least
$(i+2+l)+(i+1+l)=2i+3+2l$ boxes in the outer diagonal. Thus, $l=0$,
$n=2i+4$, and of the $2i+3$ boxes in the outer diagonal, full-sets
end in $i+2$ of them. But then one must end in $(n-1,1)$, resulting
in an $r^{+}$-full-set for $r=C(n-1,1)$.

(\ref{enu:GreaterOrEqual2iPlus43}). Let $n=2i+4+l$ for some $l\geq0$.
Let $C\in\chain_{i}(n)$ and suppose its number of plus-full-sets
is $t$. By Corollary \ref{cor:UniqueRepresentationByPhi}, there
exists a unique $\widetilde{C}\in\nofull_{i}(n-t)$ and a unique $t$-tuple
$(r_{1},r_{2},\ldots,r_{t})$, such that $C=\left(\phi_{i,n-1}^{r_{1}}\circ\phi_{i,n-2}^{r_{2}}\circ\cdots\circ\phi_{i,n-t}^{r_{t}}\right)(\widetilde{C})$.
By (\ref{enu:GreaterOrEqual2iPlus42}), $n-t\leq2i+3$, thus $t\geq l+1$.
Since the length of $C$ is $n+i=3i+4+l$, there exists a $j^{+}$-full-set
in $C$ satisfying $j\leq3i+4$. Thus, $C\in\full_{i}^{r}(n)$ for
some $r\leq j\leq3i+4$. 

We show by induction on $n$, that for each $j\in[3i+4]$, $\full_{i}^{j}(n)$
is nonempty. For the base case $n=2i+4$, let $C\in\nofull_{i}(2i+3)$.
Then for each $r$ (as in Theorem \ref{thrm:PhiINR}), satisfying
$0\leq r\leq3i+3$, $\phi_{i,2i+3}^{r}(C)\in\full_{i}^{r+1}(2i+4)$.
Now suppose the statement is true for $n$. By the inductive hypothesis,
there exists $C\in\full_{i}^{3i+4}(n)$. We may choose any value of
$r$ (as in Theorem \ref{thrm:PhiINR}), satisfying $0\leq r\leq3i+3$,
to obtain $\phi_{i,n}^{r}(C)\in\full_{i}^{r+1}(n+1)$.\end{proof}
\begin{thm}
\label{thm:MainThrm}For each $i\geq-1$ and for all $n\geq1$, the
number of maximal chains in $\tam_{n}$ of length $n+i$ is
\begin{equation}
\#\chain_{i}(n)=\sum_{t=1}^{2i+3}\binom{n+i}{t+i}\#\nofull_{i}(t),\label{eq:NumberOfCinFinalFormula}
\end{equation}
a polynomial in $n$ of degree $3i+3$. The initial values of $\#\nofull_{i}(n)$,
$n\in[2i+3]$, are 
\begin{equation}
\#\nofull_{i}(n)=\sum_{t=1}^{n}(-1)^{n-t}\binom{n+i}{t+i}\#\chain_{i}(t).\label{eq:NumberOfNinFinalFormula}
\end{equation}
\end{thm}
\begin{proof}
If $n\geq2i+4$, then for all $t$ satisfying $2i+4\leq t\leq n$,
$\#\nofull_{i}(t)=0$, thus equation (\ref{eq:=000023Ci(n)FirstExpr})
reduces to equation (\ref{eq:NumberOfCinFinalFormula}). On the other
hand, suppose $1\leq n\leq2i+3$. Then for all $t$ satisfying $n<t\leq2i+3$,
$\binom{n+i}{t+i}=0$, thus equation (\ref{eq:NumberOfCinFinalFormula})
reduces to equation (\ref{eq:=000023Ci(n)FirstExpr}). Equation (\ref{eq:NumberOfNinFinalFormula})
is equation (\ref{eq:=000023Ni(n)FirstExpr}).

The summand for $t=2i+3$ in equation (\ref{eq:NumberOfCinFinalFormula})
is the one containing the term with the largest power of $n$, the
term being 
\[
\frac{n^{3i+3}}{\left(3i+3\right)!}\#\nofull_{i}(2i+3).
\]
By Lemma \ref{lem:2iPlus3}, this term is nonzero, thus $\#\chain_{i}(n)$
is a polynomial of degree $3i+3$.
\end{proof}
\begin{table}[h]
\begin{centering}
\begin{tabular}{|c|c|c|c|c|c|c|c|c|c|c|c|c|c|}
\hline 
{\footnotesize{}Length} & {\footnotesize{}$\tam_{1}$} & {\footnotesize{}$\tam_{2}$} & {\footnotesize{}$\tam_{3}$} & {\footnotesize{}$\tam_{4}$} & {\footnotesize{}$\tam_{5}$} & {\footnotesize{}$\tam_{6}$} & {\footnotesize{}$\tam_{7}$} & {\footnotesize{}$\tam_{8}$} & {\footnotesize{}$\tam_{9}$} & {\footnotesize{}$\tam_{10}$} & {\footnotesize{}$\tam_{11}$} & {\footnotesize{}$\tam_{12}$} & {\footnotesize{}$\tam_{13}$}\tabularnewline
\hline 
\hline 
{\footnotesize{}n - 1} & {\footnotesize{}1} &  &  &  &  &  &  &  &  &  &  &  & \tabularnewline
\hline 
{\footnotesize{}n} &  &  & {\footnotesize{}1} &  &  &  &  &  &  &  &  &  & \tabularnewline
\hline 
{\footnotesize{}n + 1} &  &  &  & {\footnotesize{}2} & {\footnotesize{}10} &  &  &  &  &  &  &  & \tabularnewline
\hline 
{\footnotesize{}n + 2} &  &  &  & {\footnotesize{}2} & {\footnotesize{}8} & {\footnotesize{}112} & {\footnotesize{}280} &  &  &  &  &  & \tabularnewline
\hline 
{\footnotesize{}n + 3} &  &  &  &  & {\footnotesize{}18} & {\footnotesize{}220} & {\footnotesize{}1,464} & {\footnotesize{}9,240} & {\footnotesize{}15,400} &  &  &  & \tabularnewline
\hline 
{\footnotesize{}n + 4} &  &  &  &  & {\footnotesize{}13} & {\footnotesize{}218} & {\footnotesize{}5,322} & {\footnotesize{}42,592} & {\footnotesize{}281,424} & {\footnotesize{}1,121,120} & {\footnotesize{}1,401,400} &  & \tabularnewline
\hline 
{\footnotesize{}n + 5} &  &  &  &  & {\footnotesize{}12} & {\footnotesize{}324} & {\footnotesize{}8,052} & {\footnotesize{}142,944} & {\footnotesize{}1,714,700} & {\footnotesize{}12,180,168} & {\footnotesize{}65,985,920} & {\footnotesize{}190,590,400} & {\footnotesize{}190,590,400}\tabularnewline
\hline 
\end{tabular}
\par\end{centering}

\caption{$\#\nofull_{i}(n)$: Number of maximal chains in $\tam_{n}$ of length
$n+i$ with no plus-full-sets }
\label{tab:5-1}
\end{table}

Table \ref{tab:5-1} is a computer based compilation of the numbers
of maximal chains in $\nofull_{i}(n)$ for $-1\leq i\leq5$. The problem
of enumerating $\chain_{i}(n)$ is reduced to computing $\#\nofull_{i}(n)$,
for $n\in[2i+3]$. For example, the number of maximal chains of length
$14$ in $\tam_{11}$ is
\begin{eqnarray}
\#\chain_{3}(11) & = & 18\binom{14}{8}+220\binom{14}{9}+1464\binom{14}{10}+9240\binom{14}{11}+15400\binom{14}{12},\nonumber \\
 & = & 18\binom{14}{6}+220\binom{14}{5}+1464\binom{14}{4}+9240\binom{14}{3}+15400\binom{14}{2}.\label{eq:NumberOfCinExample}
\end{eqnarray}
According to Theorem \ref{thrm:GOrE2IPlus4}, for all $n\geq2i+4=10$,
each maximal chain in $\chain_{3}(n)$ has a plus-full-set, so this
follows for $n=11$. The interpretation for equation (\ref{eq:NumberOfCinExample})
is that in $\chain_{3}(11)$, the numbers of maximal chains having
exactly $2$, $3$, $4$, $5$ and $6$ plus-full-sets are $15400\binom{14}{2}$,
$9240\binom{14}{3}$, $1464\binom{14}{4}$, $220\binom{14}{5}$ and
$18\binom{14}{6}$, respectively. Moreover, the subset of $\chain_{3}(11)$
of maximal chains containing exactly $j$ plus-full-sets, $2\leq j\leq6$,
has an equal number of maximal chains over all $j$-element subsets
of $[14]$ of particular sets of plus-full-sets.

Interpreting maximal chains in the Tamari lattice as $\psi$-tableaux
has proven an efficient method of study. The pursuit of the formula
for $\#\chain_{i}(n)$ led to the plus-full-set property and some
interesting combinatorics. Based on numerical evidence, we conclude
this note with a conjecture. 
\begin{conjecture}
For all $i\geq-1$,
\begin{equation}
\#\nofull_{i}(2i+3)=\prod_{j=1}^{i+1}\binom{3j-1}{2},\label{eq:Conjecture1}
\end{equation}
and for all $i\geq0$,
\begin{equation}
\#\nofull_{i}(2i+2)=\frac{i}{5}\prod_{j=1}^{i+1}\binom{3j-1}{2}.\label{eq:Conjecture2}
\end{equation}

\end{conjecture}

\section*{Acknowledgments}

Many thanks goes out to Kevin Treat for proofreading and
helping out with notation herein. I thank Susanna Fishel for suggesting
problems on the Tamari lattice and for her advice and encouragement.
I was partially supported by GAANN Grant $\#$P200A120120.

\bibliographystyle{plain}
\bibliography{LNTamariRecursion}

\end{document}